\documentclass[10pt,a4paper]{article}
\usepackage[utf8]{inputenc}
\usepackage[english]{babel}
\usepackage{amsmath}
\usepackage{amsfonts}
\usepackage{amssymb}
\usepackage{graphicx}
\usepackage{color}
\newcommand{\al}{\alpha}
\newcommand{\sgn}{\mathop{\mathrm{sgn}}}

\usepackage{epstopdf} 

\newtheorem{thm}{Theorem}[section]
\newtheorem{cor}[thm]{Corollary}

\newtheorem{rem}[thm]{Remark}

\newtheorem{lemma}[thm]{Lemma}

\newenvironment{proof}{\begin{trivlist} \item[] {\em Proof:}}{\hfill $\Box$
                       \end{trivlist}}

\newenvironment{proofthm}{\begin{trivlist} \item[] {\em Proof of the Theorem:}}{\hfill $\Box$
                       \end{trivlist}}

\allowdisplaybreaks[2]

\begin{document}

\title{Remarks on geometric properties of SQG sharp fronts and $\alpha$-patches }
\author{Angel Castro, Diego C\'ordoba, \\ Javier G\'omez-Serrano, Alberto Mart\'in Zamora}

\maketitle

\begin{abstract}

Guided by numerical simulations, we present the proof of two results concerning the behaviour of SQG sharp fronts and $\alpha$-patches. We establish that ellipses are not rotational solutions and we prove that initially convex interfaces may lose this property in finite time.\\


\end{abstract}

\section{Introduction}

Our goal in this article is to investigate some simple questions about the evolution of the surface quasi-geostrophic (SQG) front, some of which are well known for the vortex patch problem. The surface quasi-geostrophic equation was introduced in the mathematical community by Constantin, Majda and Tabak in \cite{Constantin-Majda-Tabak:formation-fronts-qg}. This equation is derived considering small Rossby and Ekman numbers and constant potential vorticity. It provides a mathematical description of the evolution of the temperature from a general quasi-geostrophic system for atmospheric and oceanic flows (see  \cite{Pedlosky:geophysical} for more details).

 Leaving aside its interest from a geophysical point of view, the SQG equation is a two dimensional model of the principal equations of inviscid fluid dynamics in 3D, namely, the incompressible Euler equations. The SQG system is the following set of equations:

\begin{equation}
\left\{
\begin{array}{rl}
\partial_t \theta (x,t)+ u(x,t) \cdot \nabla \theta(x,t)& = 0, \ \ (x,t) \in \mathbb{R}^2 \times \mathbb{R}_+\\
  u(x,t) & = \nabla^\perp \Psi (x,t)\\
\theta (x,t) & = \Lambda \Psi (x,t)\\
\end{array}
\right.
\end{equation}

	where $\Lambda = (-\Delta)^{\frac{1}{2}}$ and $\Psi$ is the stream function.
		
 The first equation simply represents the fact that the temperature, $\theta$, is advected by the velocity $u$. The second and third equations relate the temperature, $\theta$, with the velocity, $u$, through a non local operator. Those two equations can be rewritten as $ u = ( - R_2 \theta, R_1 \theta)$ where $R_i$ are the Riesz Transforms, given in $\mathbb{R}^2 $ by
 \begin{align*}
 R_i (f)(x) = \frac{1}{2\pi} \int_{\mathbb{R} ^2} \frac{ (x_i - y_i) f(y)} { \vert x - y \vert^3} dy.\\
\end{align*}

Motivated by the articles \cite{Constantin-Majda-Tabak:formation-fronts-qg} and \cite{Held-Pierrehumbert-Garner-Swanson:sqg-dynamics}, a lot of effort has been devoted to understanding these equations. In particular, the problem of whether the SQG system presents finite time singularities or not is open.

  The existence of global weak solutions in $L^2$ was proven by Resnick in \cite{Resnick:phd-thesis-sqg-chicago} using an extra cancellation due to the oddness of the Riesz transform. A particular kind of weak solution for an active scalar are the so called \textit{patches}, i.e., solutions for which $\theta$ is
 a step function:


  \begin{align}
  \theta(x) =
  \left\{
 \begin{array}{ll}
   \theta_1, \text{ if } \ \ x \in \Omega(t) \\ 
   \theta_2, \text{ if }  \ \ x \in \Omega(t) ^c, \\
   \end{array}
  \right.
\end{align}

 where $ \Omega(0)$ is given by the initial distribution of $\theta$, and $\Omega(t)$ is the evolution of $\Omega(0)$ under the velocity field $u$ given by $u(x,t)  = \nabla^\perp \Lambda^{-1} \theta (x,t) $.

  The evolution of such distribution is completely determined by the evolution of the boundary, allowing the problem to be treated as a non-local one dimensional equation for the contour of $\Omega(t)$. In this setting, local existence of smooth solutions was first obtained for analytic curves by Rodrigo in \cite{Rodrigo:evolution-sharp-fronts-qg}. The question of local existence of simply connected  \textit{patches} with Sobolev regularity of its boundary was addressed by Gancedo in \cite{Gancedo:existence-alpha-patch-sobolev}.

 To get a better understanding of the behaviour of solutions of these interface problems, several numerical experiments have been performed. In \cite{Cordoba-Fontelos-Mancho-Rodrigo:evidence-singularities-contour-dynamics} the problem of the evolution of two patches  is studied. Their simulations suggest an asymptotically self-similar singular scenario in which the distance between both patches goes to zero in finite time while simultaneously the curvature of the boundaries blows up.   Recently, Gancedo and Strain \cite{Gancedo-Strain:absence-splash-muskat-SQG} proved that in fact, no splash singularity can be formed, i.e., two interfaces can not collapse in a point, if the interfaces remain smooth.

 Also in \cite{Cordoba-Fontelos-Mancho-Rodrigo:evidence-singularities-contour-dynamics}, a new family of patch problems interpolating between the vortex patch and the SQG patch is introduced: the $\alpha $-patch model. For the vortex patch problem, the velocity field is obtained as $ v = \nabla ^\perp \Delta^{-1} \theta$, where $\theta$ is the vorticity, while for SQG the velocity is given by $ v = \nabla^\perp \Delta^{-\frac{1}{2}} \theta.$ The velocity of the $\alpha$-patch is defined by $ v = \nabla^\perp \Delta^{- (1 - \frac{\alpha}{2})} \theta$, for $ \alpha \in (0,2)$.  The local existence proof for $\alpha \in (1,2)$ can be found in \cite{Chae-Constantin-Cordoba-Gancedo-Wu:gsqg-singular-velocities}.

 The evolution equation for the interface of an $\alpha-$ patch, which we parametrize as a $2 \pi$ periodic curve $z(x)$, can be written as
 \begin{align}
\label{Ecuacion-alpha-patch}
 \partial_t z(x,t) = -\frac{(\theta_2 - \theta_1) \Gamma(\frac{\alpha}{2})}{\pi^2 2^{2-\al}\Gamma(\frac{2-\al}{2})}  \int_{0} ^{ 2 \pi} \frac{ \partial_x z (x,t) - \partial_x z(x-y,t) }{ \vert z(x,t) - z(x-y, t) \vert^{\alpha}} dy
 \end{align}

if $ 0 < \alpha < 2$, and as

\begin{align}
 \partial_t z(x)  = \frac{(\theta_2 - \theta_1)}{2 \pi}\int_0 ^{2 \pi} \log(|z(x) - z(x-y)|) \partial_x z(x-y)dy
\label{Ecuation-Vortex-Patch}
\end{align}

for the vortex patch, due to the fact that the evolution of the boundary of the patch is unchanged if the velocity is modified by adding a tangential term to the boundary, since that amounts just to a change of parametrization of the interface. For more details see \cite{Gancedo:existence-alpha-patch-sobolev}, \cite{Rodrigo:evolution-sharp-fronts-qg}. We make use of this property in our numerical simulation, adding a tangential term in order to keep the modulus of the tangent vector constant in space.


 In  \cite{Scott-Dritschel:self-similar-sqg}, based on numerical simulations, it is suggested that an elliptical patch with a big aspect ratio between its axes may develop a self-similar singularity with an explosive growth of the curvature. In \cite{Scott:scenario-singularity-quasigeostrophic}, it was already pointed out that small perturbations of thin strips may lead to a self similar cascade of instabilities, leading to a possible arc chord blow up.

	Beyond that, very little is known about the qualitative behaviour of patch like solutions of the SQG equation.

	 The analogous problem for the vorticity formulation of 2D Euler $(\alpha = 0)$ is better understood. The global existence and uniqueness of weak solutions of the 2D Euler in vorticity formulation is due to Yudovich \cite{Yudovich:Nonstationary-ideal-incompressible}. Regularity preservation for $\mathcal{C}^{1, \alpha}$ patches was obtained by Chemin using techniques from paradifferential calculus in \cite{Chemin:persistance-structures-fluides-incompressibles}. Another proof of that result, which highlights the extra cancellation on semi spheres of even kernels, can be found in \cite{Bertozzi-Constantin:global-regularity-vortex-patches} by Bertozzi and Constantin.

  In recent years, Denisov has studied the process of merging for the vortex patch problem. This is the scenario showed by the numerics of \cite{Cordoba-Fontelos-Mancho-Rodrigo:evidence-singularities-contour-dynamics} for the alpha-patch. However, for the vortex patch problem, the collapse in a point can not happen in finite time, the distance between the two patches can decay at most as fast as a double exponential. Denisov proves, in  \cite{Denisov:sharp-corner-euler-patches}, that this bound is sharp if one is allowed to modify slightly the velocity by superimposing a smooth background incompressible flow.

	In addition, it is known that there exist several solutions which evolve by rotating with constant angular velocity around its center of mass. The ellipses are the most well known example of solutions exhibiting this behaviour, and in fact, the only simply connected ones for which there are explicit formulae. The existence of that kind of solutions can be proven by a bifurcation analysis from the circular solution. For more details, see \cite{Hmidi-Mateu-Verdera:rotating-vortex-patch} and references therein.

Nowadays, the bigger computation capacity of computers has lead to their use as a mathematical tool. However, floating-point operations can result in numerical errors. To deal with this matter and be able to prove rigorous results, we use the so-called \emph{interval arithmetics}, in which instead of working with arbitrary real numbers, we perform computations over intervals which have representable numbers as endpoints. On these objects, an arithmetic is defined in such a way that we are guaranteed that for every $x \in X, y \in Y$
\begin{align*}
x \star y \in X \star Y,
\end{align*}

for any operation $\star$. For example,
\begin{align*}
[\underline{x},\overline{x}] + [\underline{y},\overline{y}] & = [\underline{x} + \underline{y}, \overline{x} + \overline{y}] \\
[\underline{x},\overline{x}] \times [\underline{y},\overline{y}] & = [\min\{\underline{x}\underline{y},\underline{x}\overline{y},\overline{x}\underline{y},\overline{x}\overline{y}\},\max\{\underline{x}\underline{y},\underline{x}\overline{y},\overline{x}\underline{y},\overline{x}\overline{y}\}]
\end{align*}
We can also define the interval version of a function $f(X)$ as an interval $I$ that satisfies that for every $x \in X$, $f(x) \in I$. Rigorous computation of integrals has been theoretically developed since the seminal work of Moore and many others \cite{Berz-Makino:high-dimensional-quadrature,Kramer-Wedner:adaptive-gauss-legendre-verified-computation,Lang:multidimensional-verified-gaussian-quadrature,Moore-Bierbaum:methods-applications-interval-analysis,Tucker:validated-numerics-book}, and has had applications in physics and fluid mechanics \cite{Holzmann-Lang-Schutt:gravitation-verified-quadrature,GomezSerrano-GraneroBelinchon:turning-muskat-computer-assisted}. In order to perform them we used the C-XSC library \cite{CXSC}.

Guided by new numeric simulations of the evolution of elliptical patches, we establish the following two results for the general $\alpha$-patch:

\begin{itemize}

\item The ellipses are not rotating patches for the equation \eqref{Ecuacion-alpha-patch} with  $0 < \alpha < 2$.

\item For every $ 0 < \alpha < 2$, there exists a solution of the $\alpha$-patch equation that begins convex, and after a finite time it is no longer convex. Moreover, there exists a solution of the vortex patch problem that has the same property.
\end{itemize}

The proof of those two results relies on establishing a sign for certain integrals. For the first theorem, we can do it by performing certain manipulations in the integral. However, for the second theorem we recur to a computer assisted proof.

The article is organized as follows: In the second section we discuss our numerical experiments. Section 3 presents a simple proof of the fact that ellipses are not rotating solutions of the $\alpha$-patch problem for $ \alpha > 0$. This is followed by a computer assisted proof of the loss of convexity in Section 4. \\
		
\section{Numerical Simulations}

In order to obtain insight on how  the ellipses evolve under SQG dynamics we performed simulations for elliptical SQG patches with varying aspect ratio. From these simulations three main differences with the behaviour under vortex patch evolution are easily appreciated:

\begin{itemize}
\item The ellipses are not rotating solutions.
\item Initially convex patches may lose that property.
\item Possible singularity formation.
\end{itemize}

For ellipses with an aspect ratio close to 1, the evolution is similar to the one of the vortex patch problem, since all points at the boundary behave in similar way. However, as the aspect ratio is increased, the fact that the ellipses do not rotate and lose convexity becomes clearer.


\begin{centering}
\begin{figure}[h!]
\centering
\includegraphics[width=0.45\textwidth]{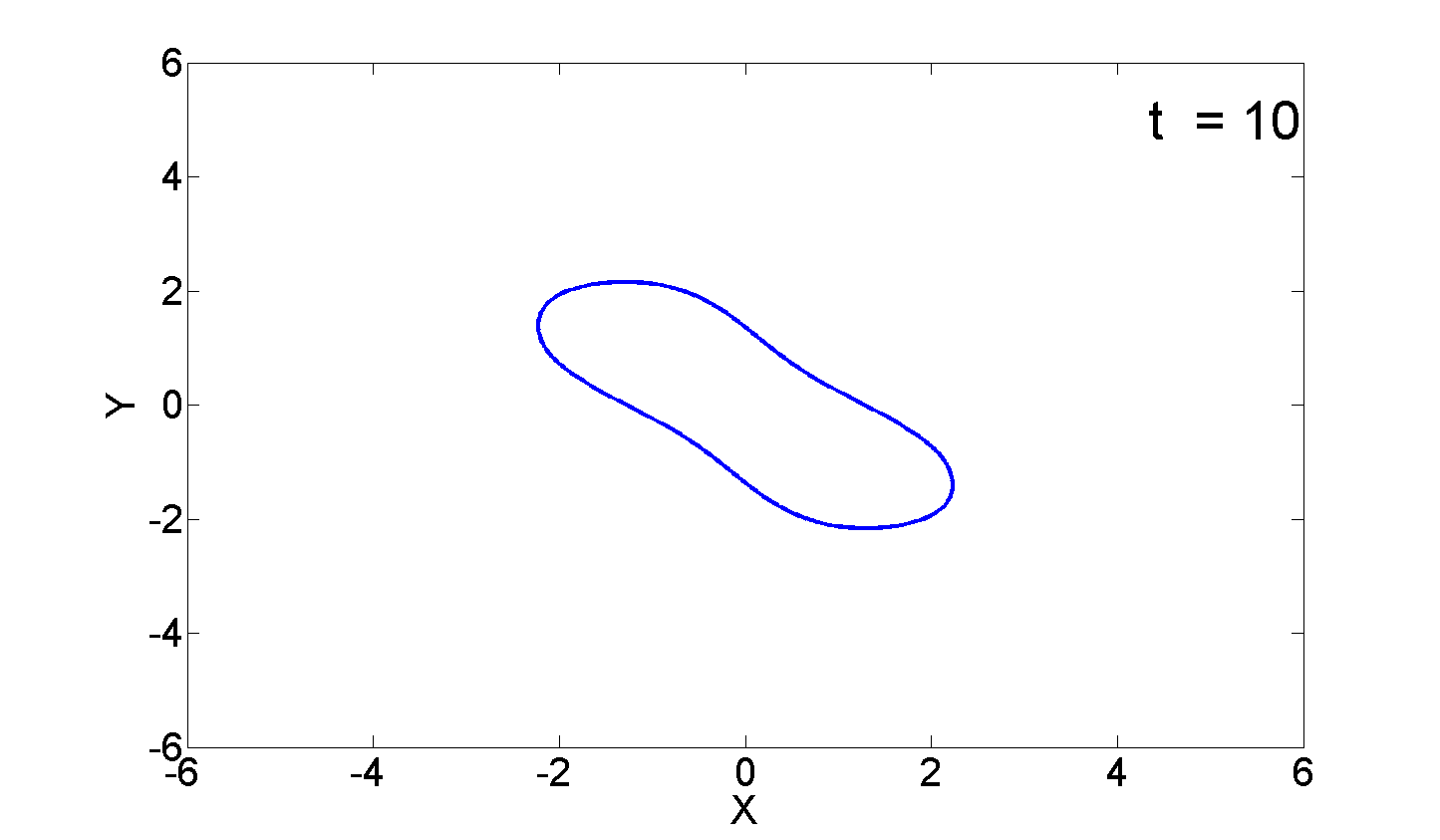}
\includegraphics[width=0.45\textwidth]{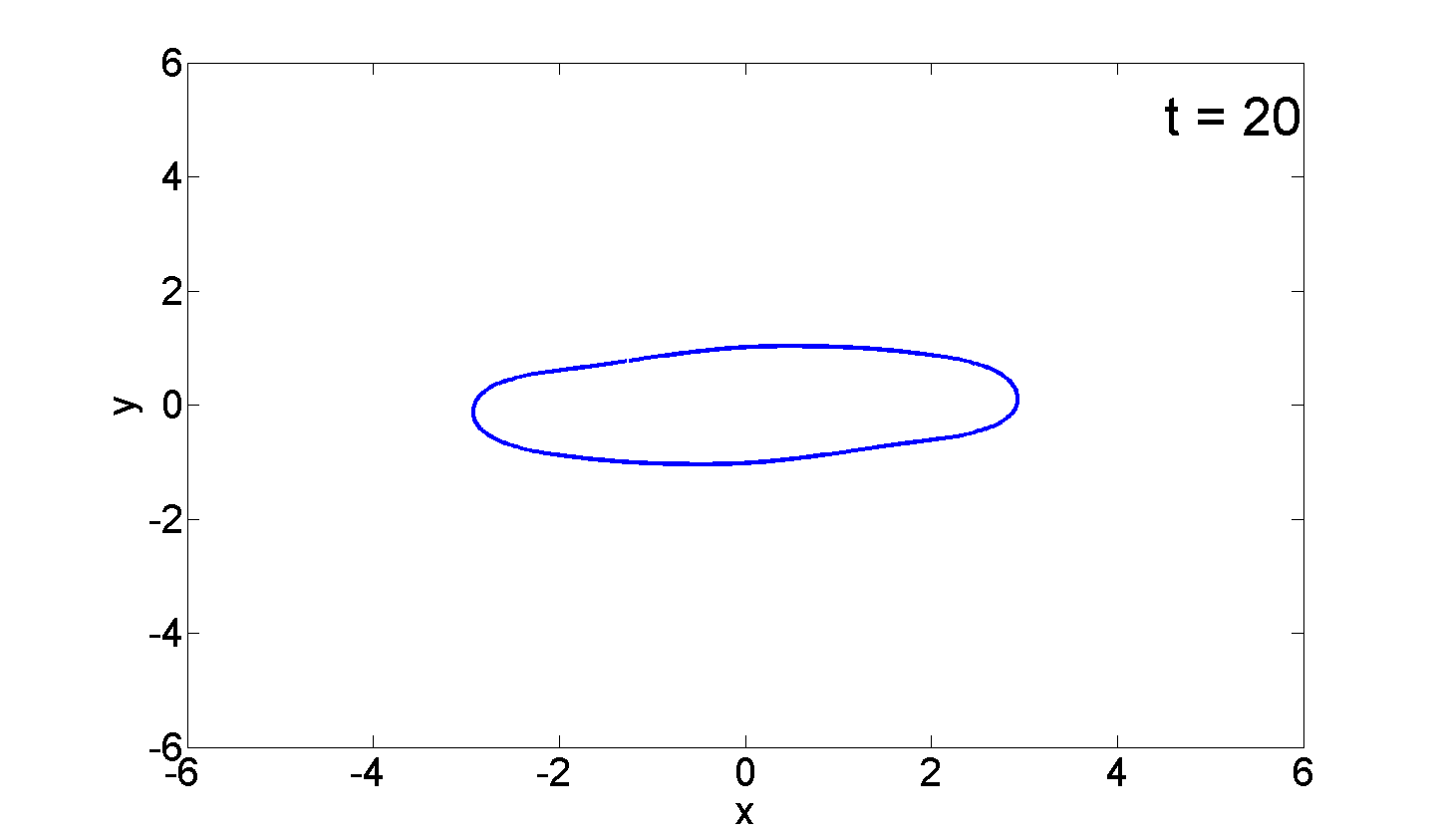}\\
\centering
\includegraphics[width=0.45\textwidth]{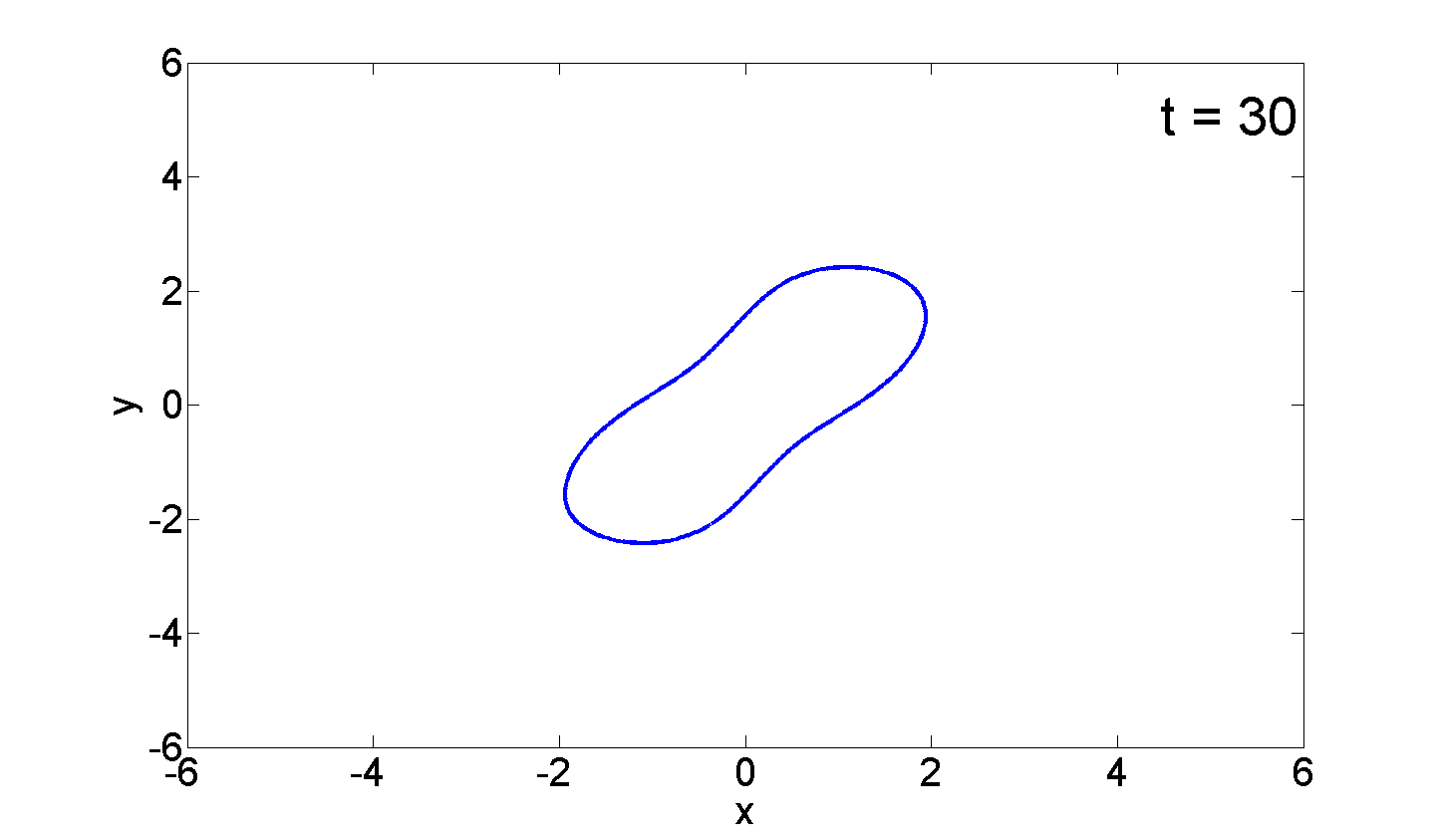}
\includegraphics[width=0.45\textwidth]{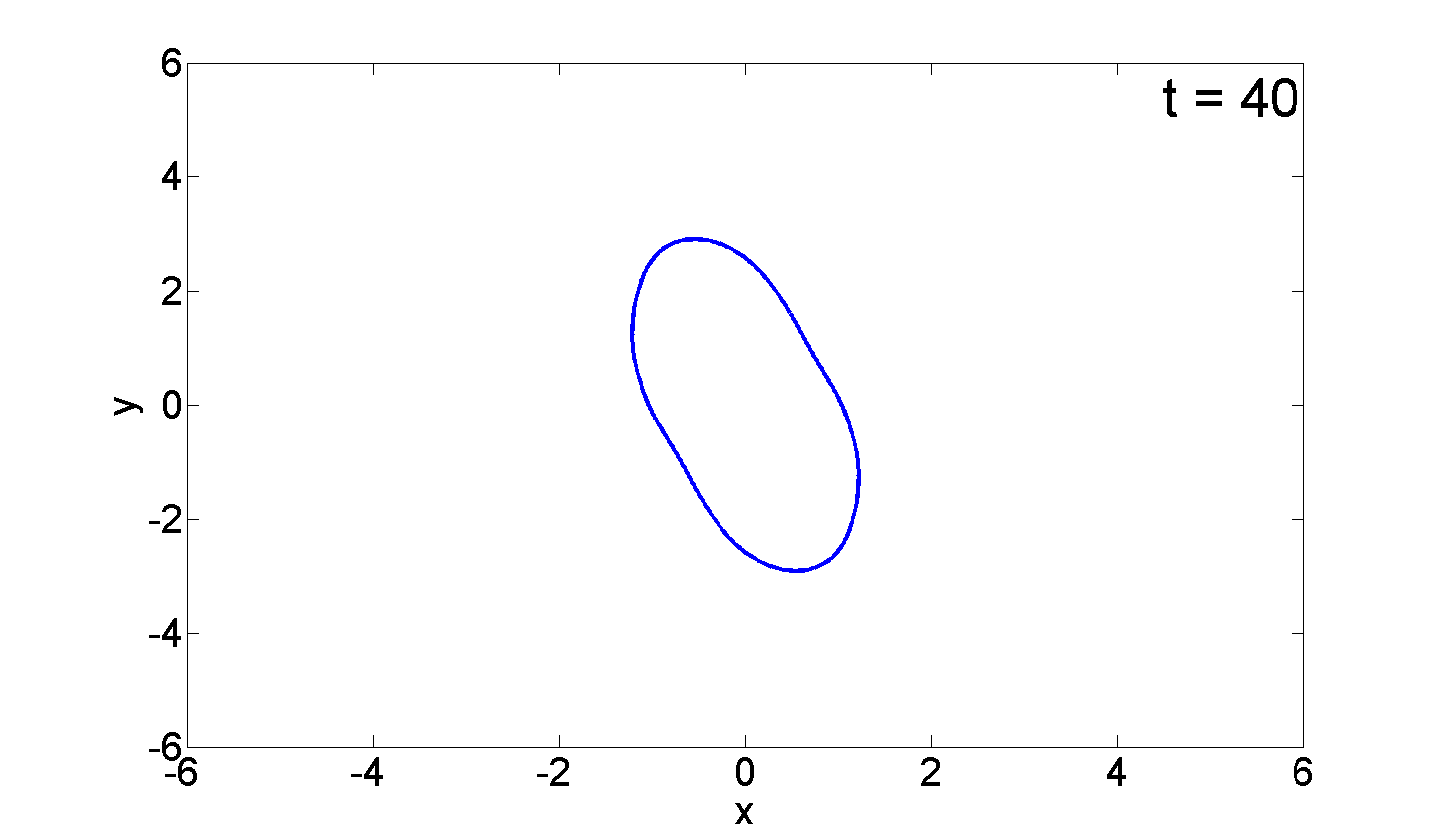}\\
\label{Elipse13 t=10}
\caption{Semiaxes  1-3, \ $ t=10,20,30,40$}
\end{figure}
\end{centering}

Figures 1-3 represent four timestamps for the evolution of the ellipses with principal semiaxes 1-3, 1-5 and 1-6. The patch with initial data $ z(x) = (\cos( x), 3 \sin(x)) $ displays a combination of a rotating motion with a smaller scale oscillation which leads to loss of convexity.

When the bigger semiaxis is increased to 5, $z(x) = ( \cos( x), 5 \sin(x))$, the patch exhibits an almost periodic behaviour in which the arc chord, i.e., the distance between different points at the boundary, decreases but eventually increases again. In addition, the loss of convexity is evident. See Figure 2.


\begin{figure}[h!]
\centering
\includegraphics[width=0.45\textwidth]{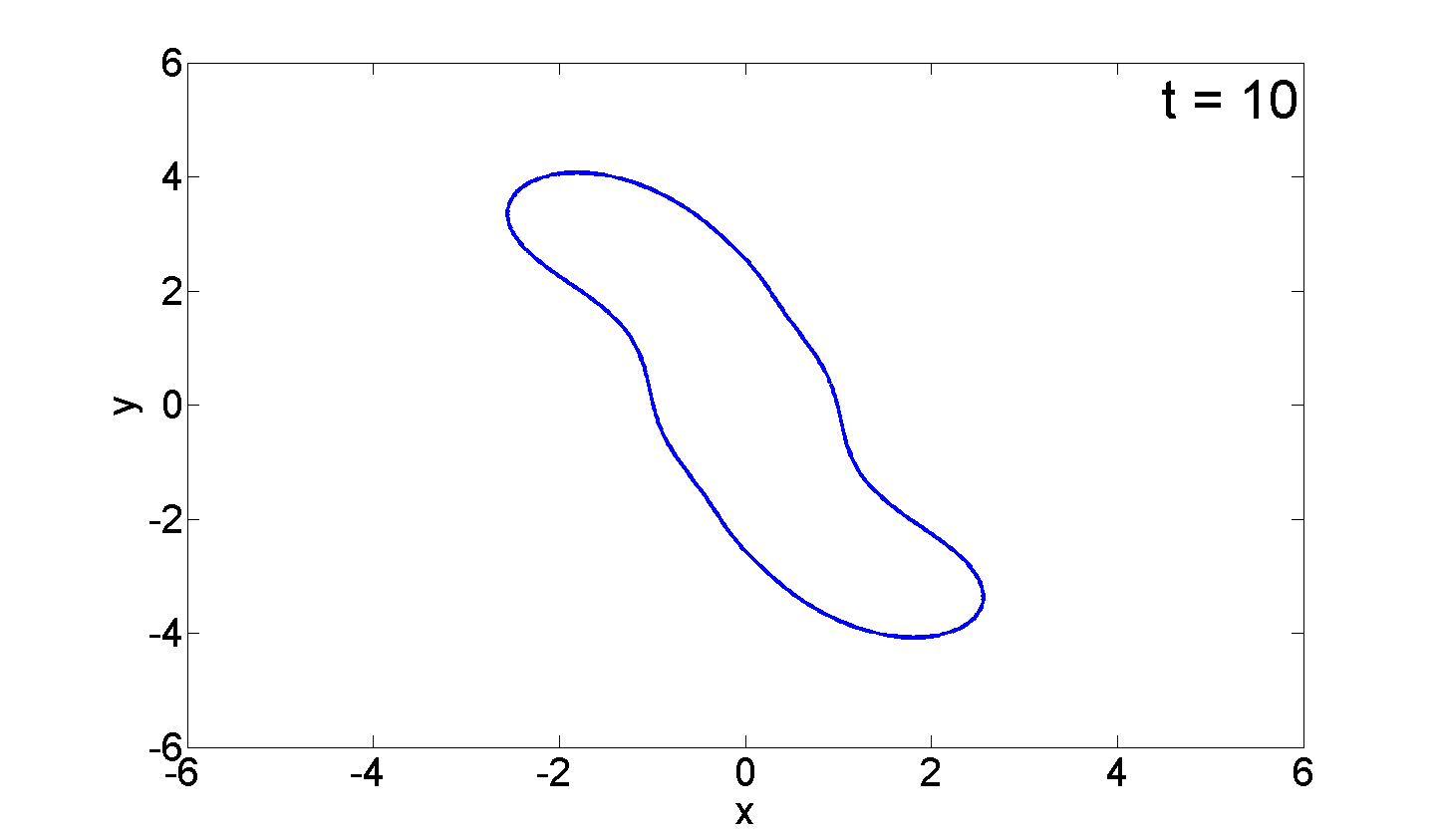}
\includegraphics[width=0.45\textwidth]{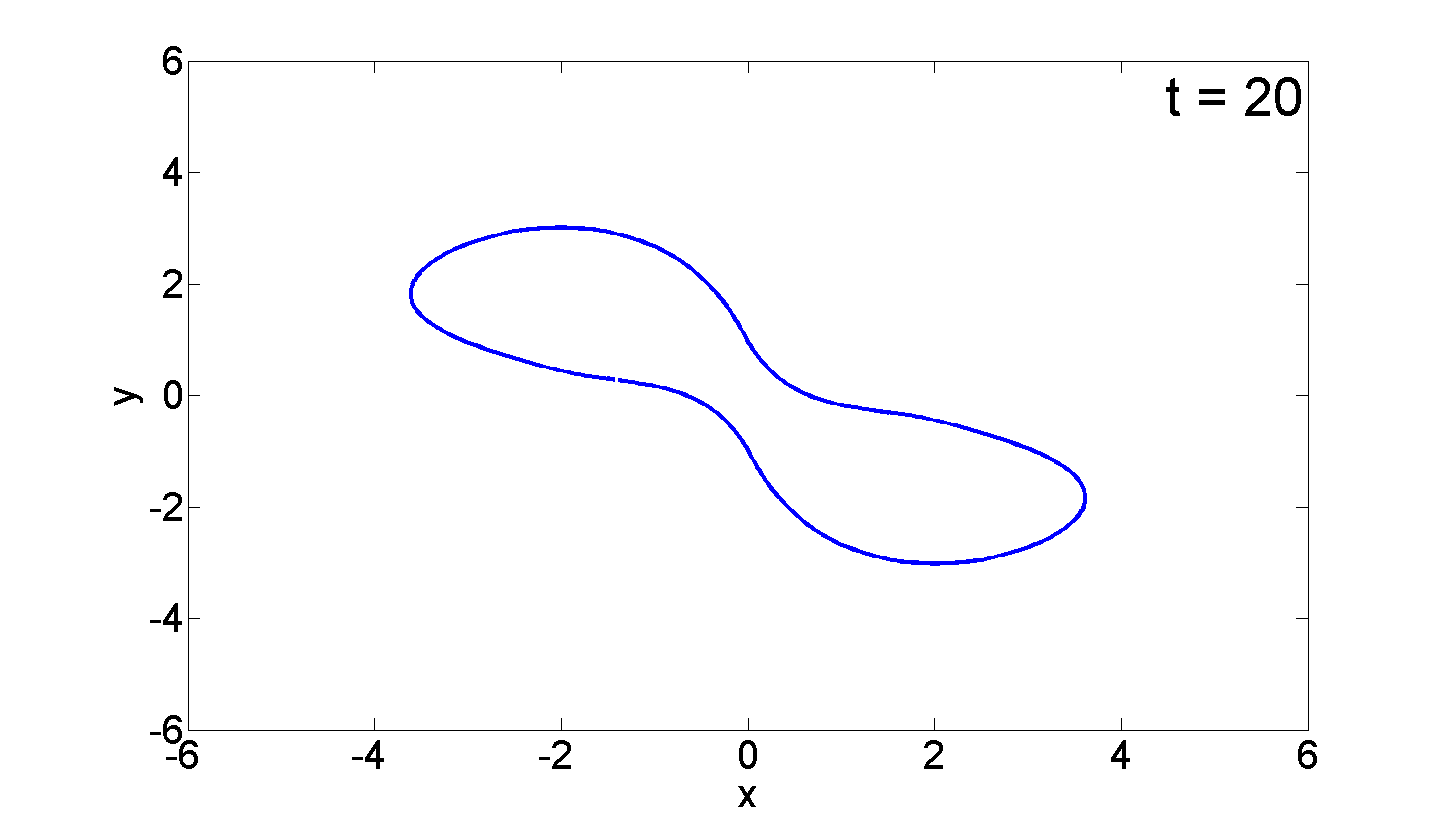}\\
\centering
\includegraphics[width=0.45\textwidth]{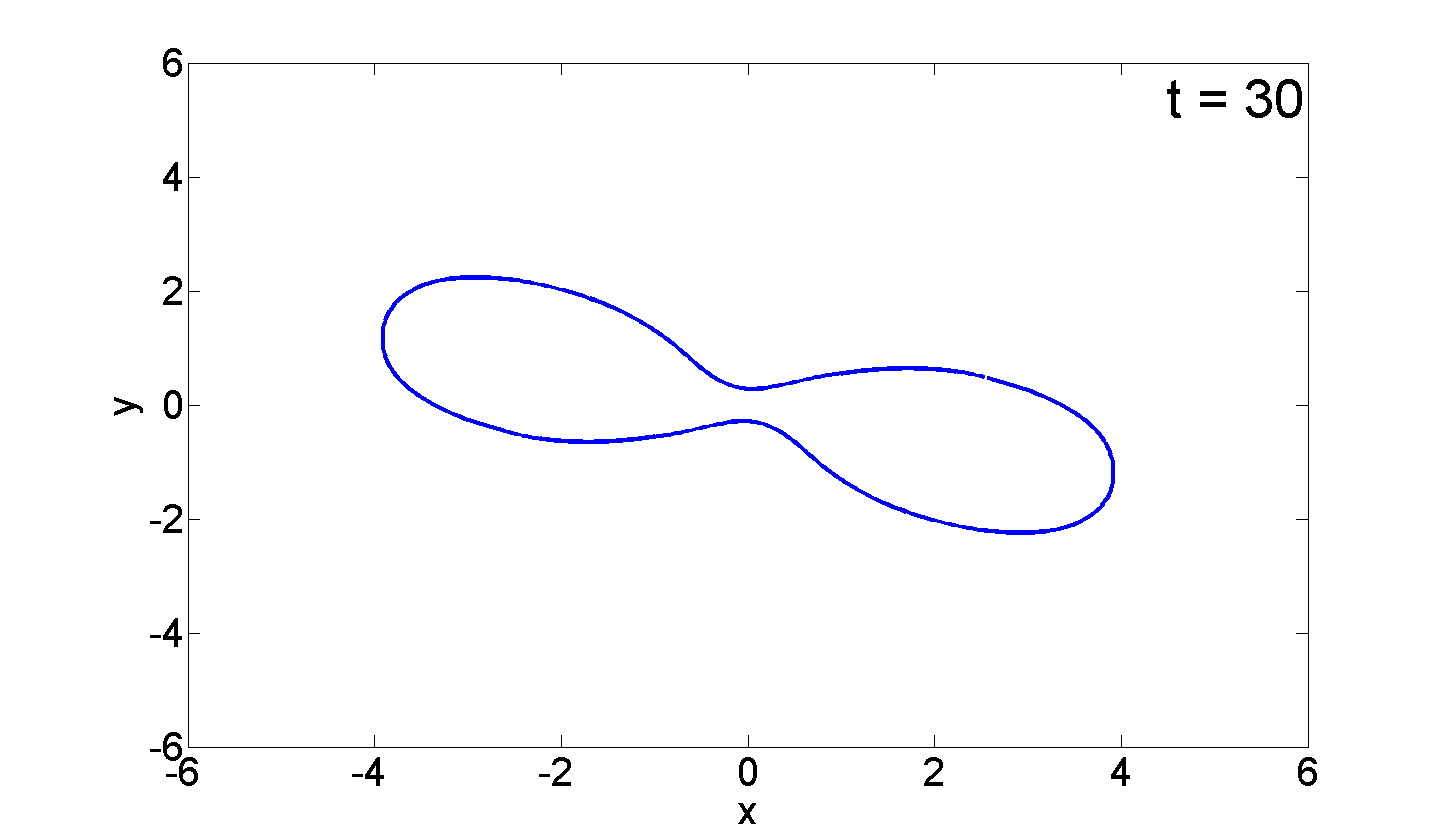}
\includegraphics[width=0.45\textwidth]{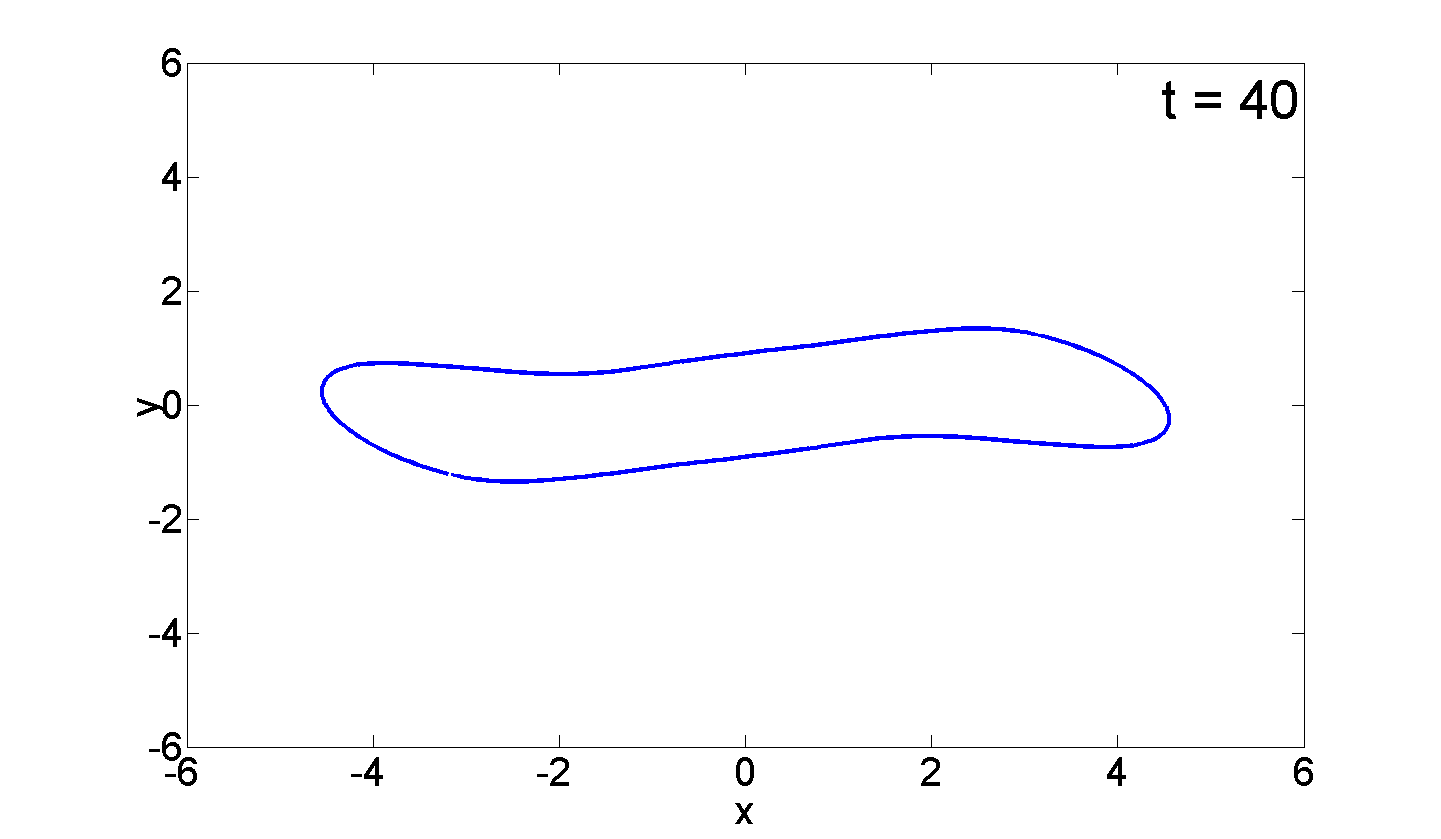}\\
\caption{Semiaxes 1-5,\ $ t=10,20,30,40$}
\end{figure}

For the ellipses with a bigger difference between its principal semiaxes, the simulations suggest that the arc chord goes to zero in finite time, developing a singularity. This question has already been addressed in  \cite{Scott-Dritschel:self-similar-sqg} and \cite{Scott:scenario-singularity-quasigeostrophic}, where more precise numerical simulations are discussed which suggest a self similar behaviour. We remark that our simulations are here used as a tool to gain intuition, not as a final purpose themselves and are by no means of a comparable accuracy to those in \cite{Scott-Dritschel:self-similar-sqg,Scott:scenario-singularity-quasigeostrophic}.


\begin{figure}[h!]
\centering
\includegraphics[width=0.45\textwidth]{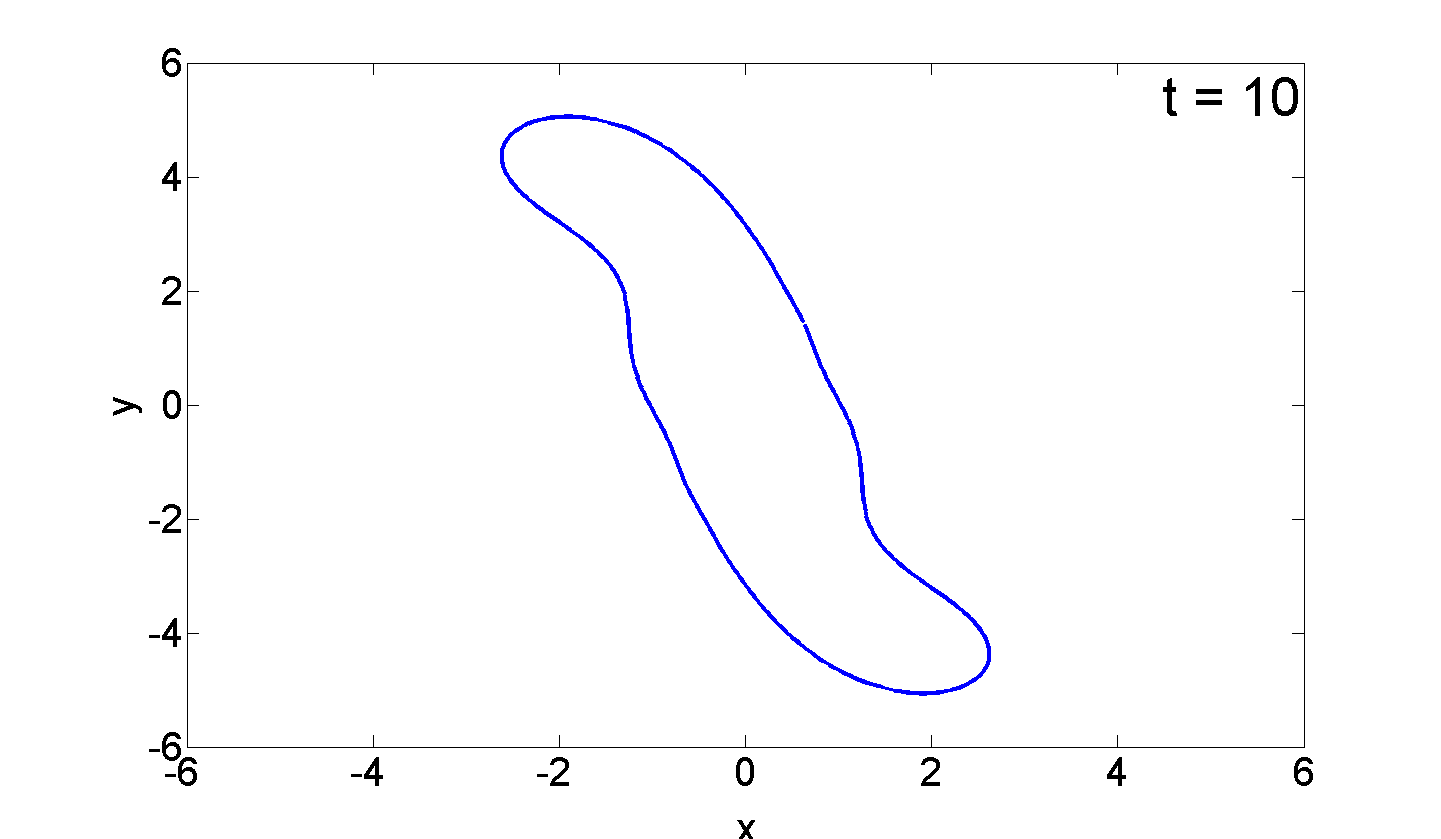}
\includegraphics[width=0.45\textwidth]{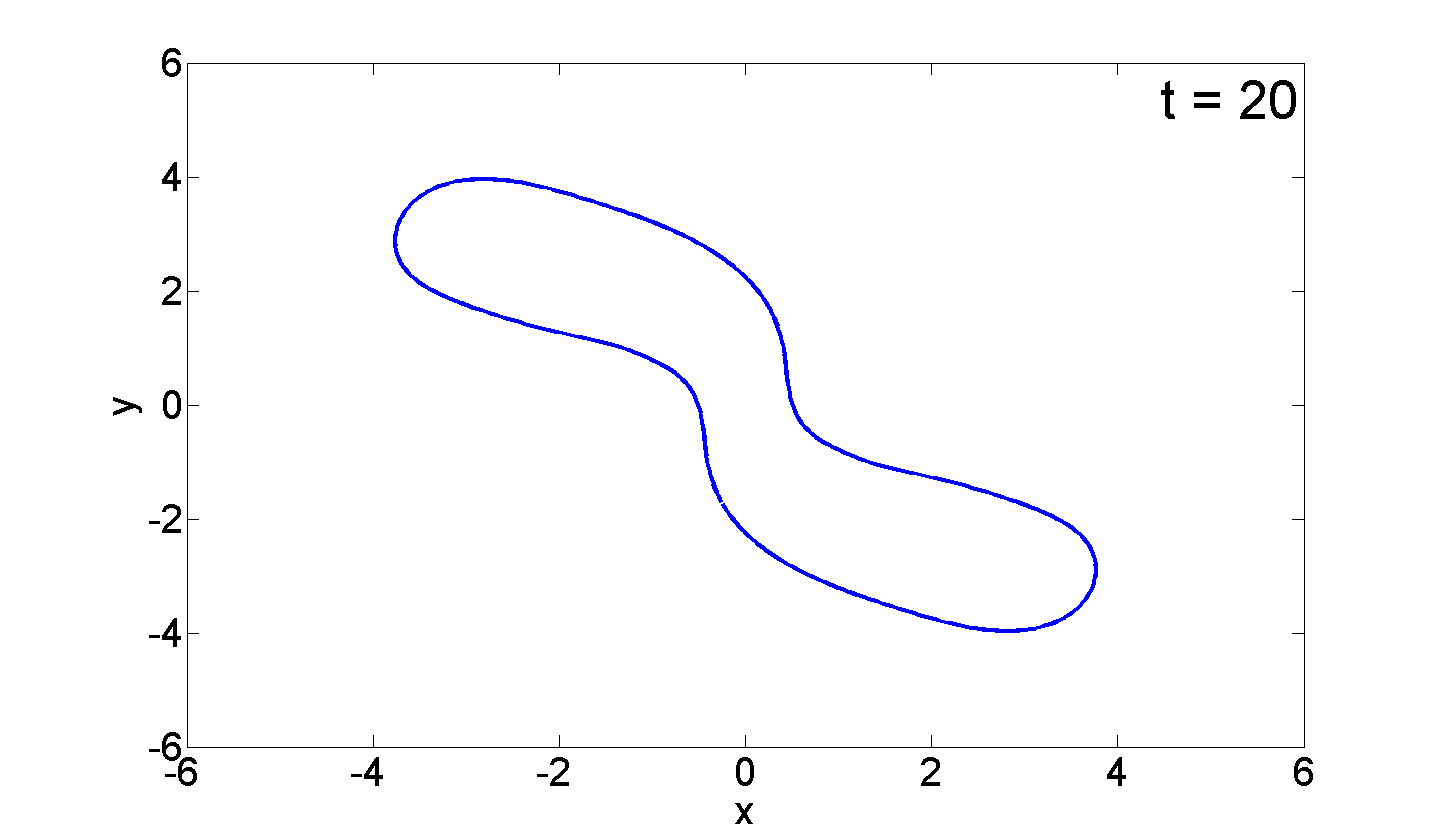}\\
\centering
\includegraphics[width=0.45\textwidth]{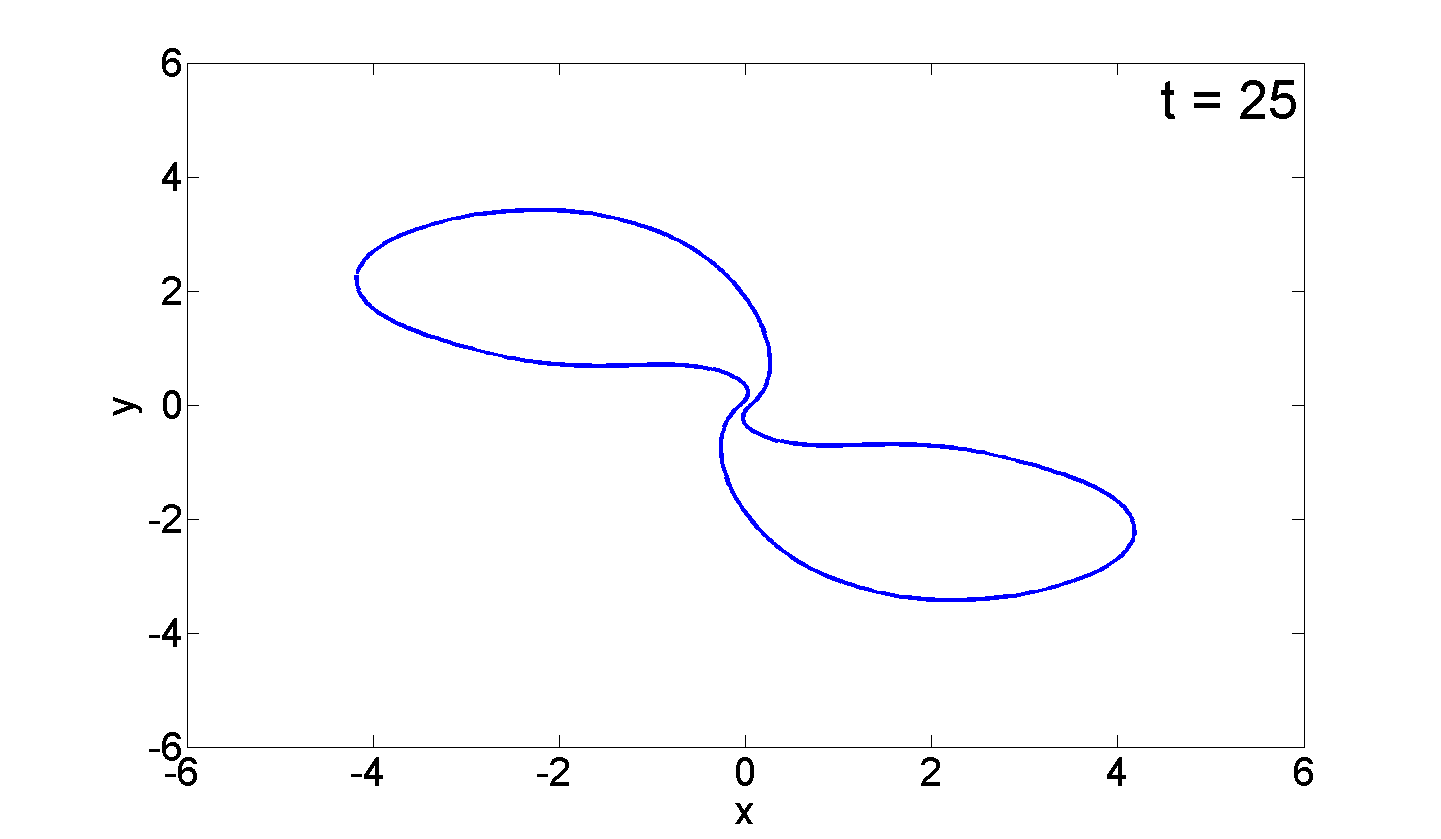}
\includegraphics[width=0.45\textwidth]{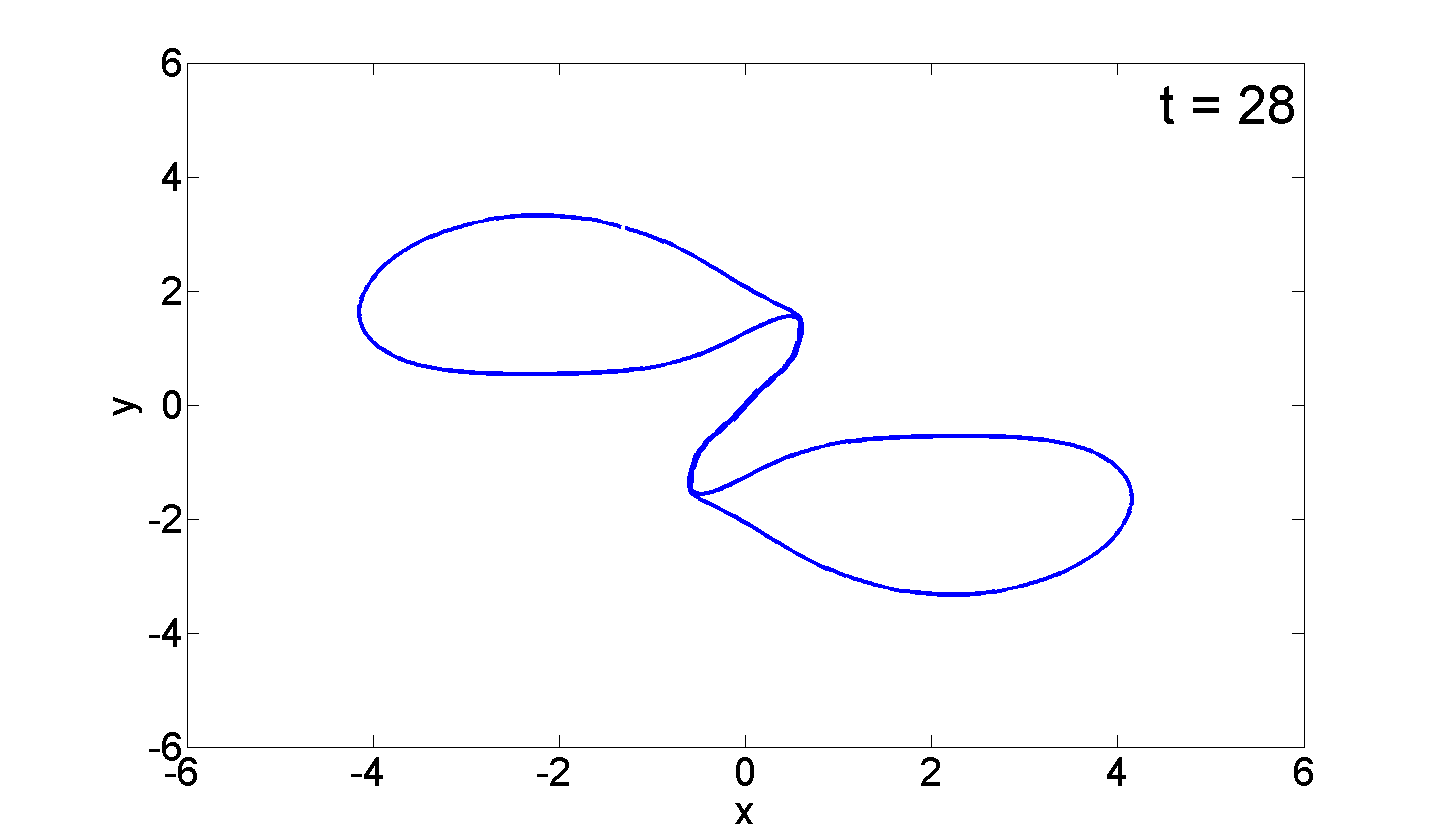}\\
\label{Elipse16}
\caption{Semiaxes 1-6,\ $ t=10,20,25,28$}
\end{figure}



\begin{figure}[h!]
\centering
\includegraphics[width=0.45\textwidth]{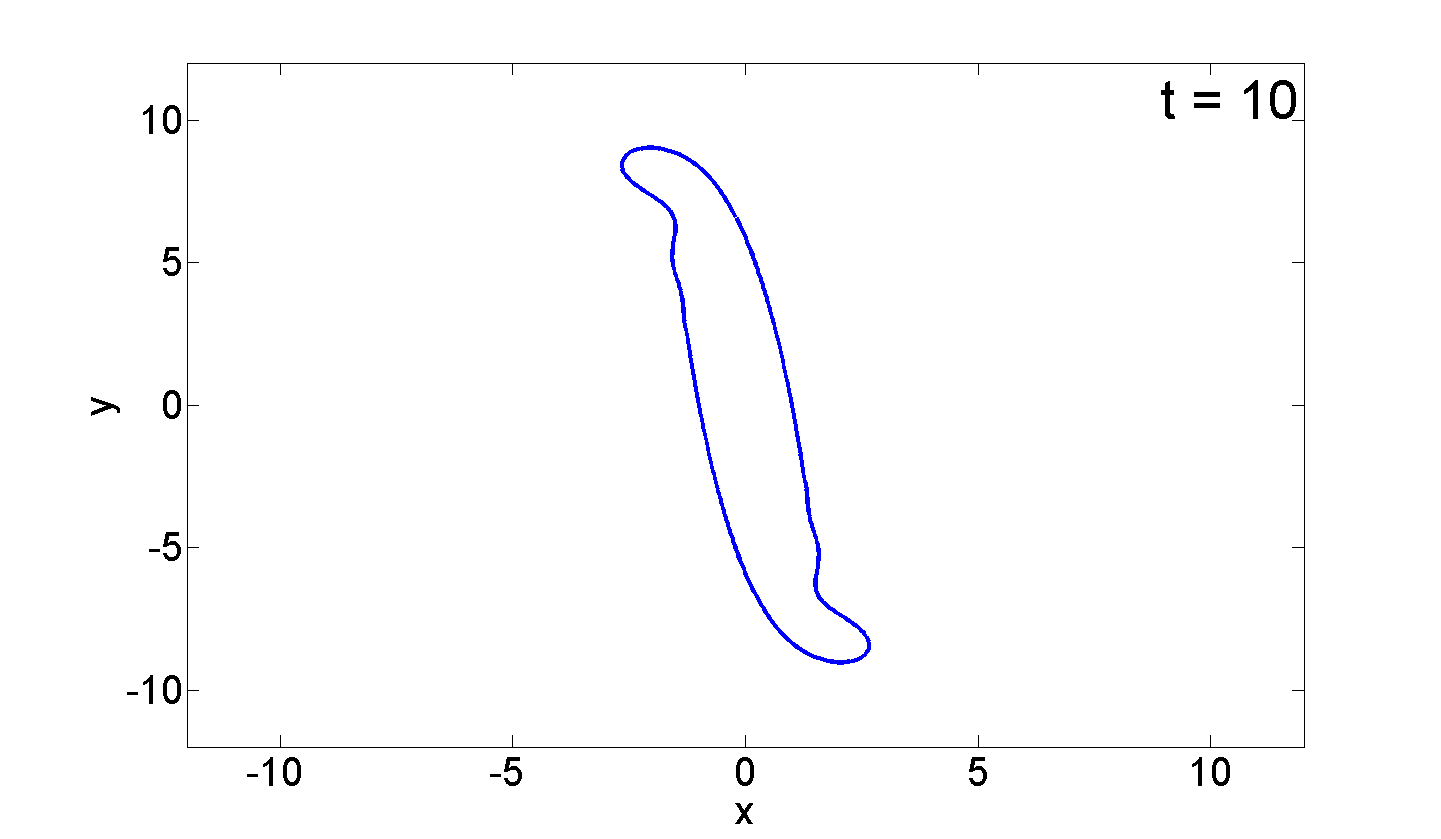}
\includegraphics[width=0.45\textwidth]{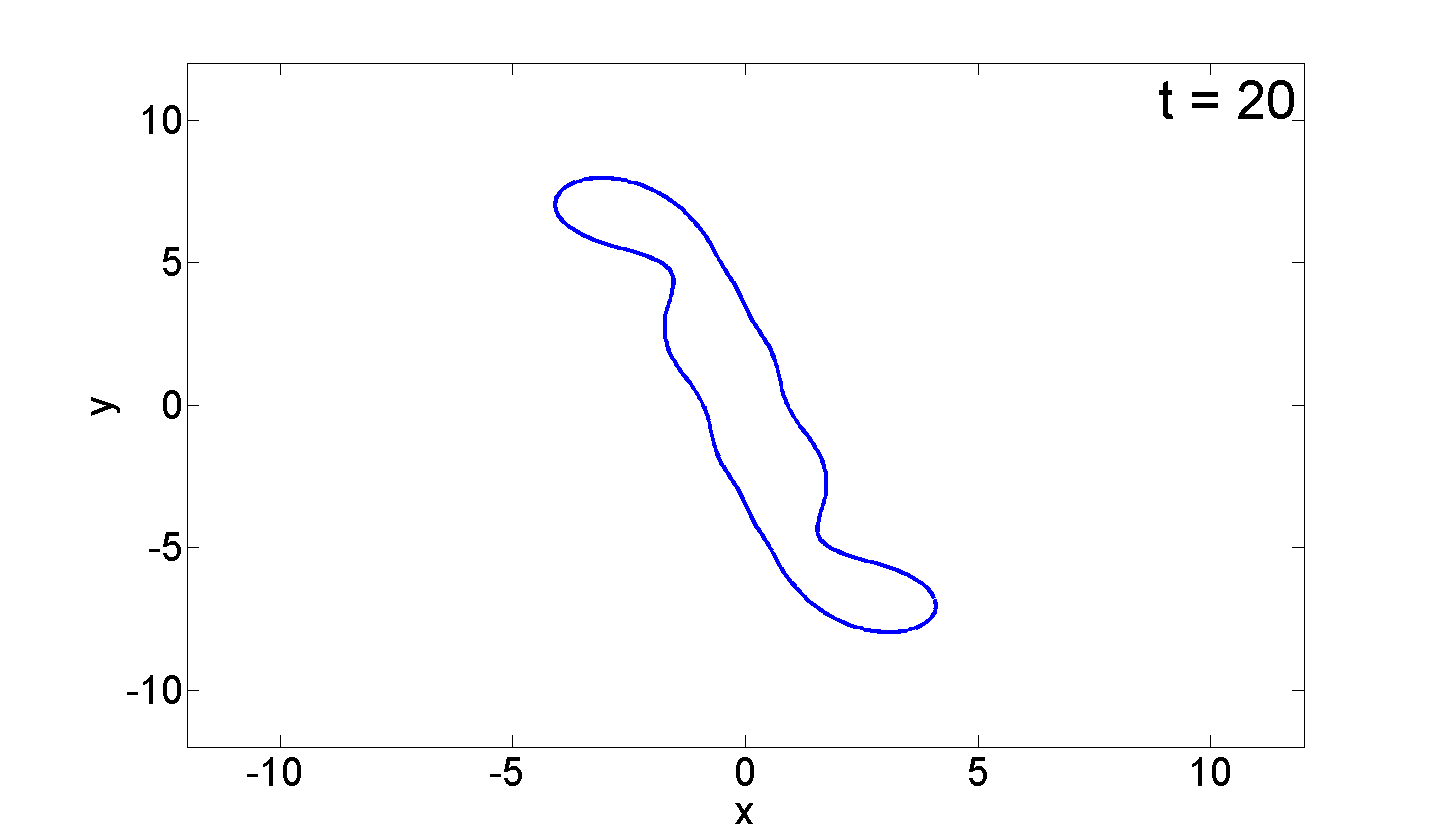}\\
\centering
\includegraphics[width=0.45\textwidth]{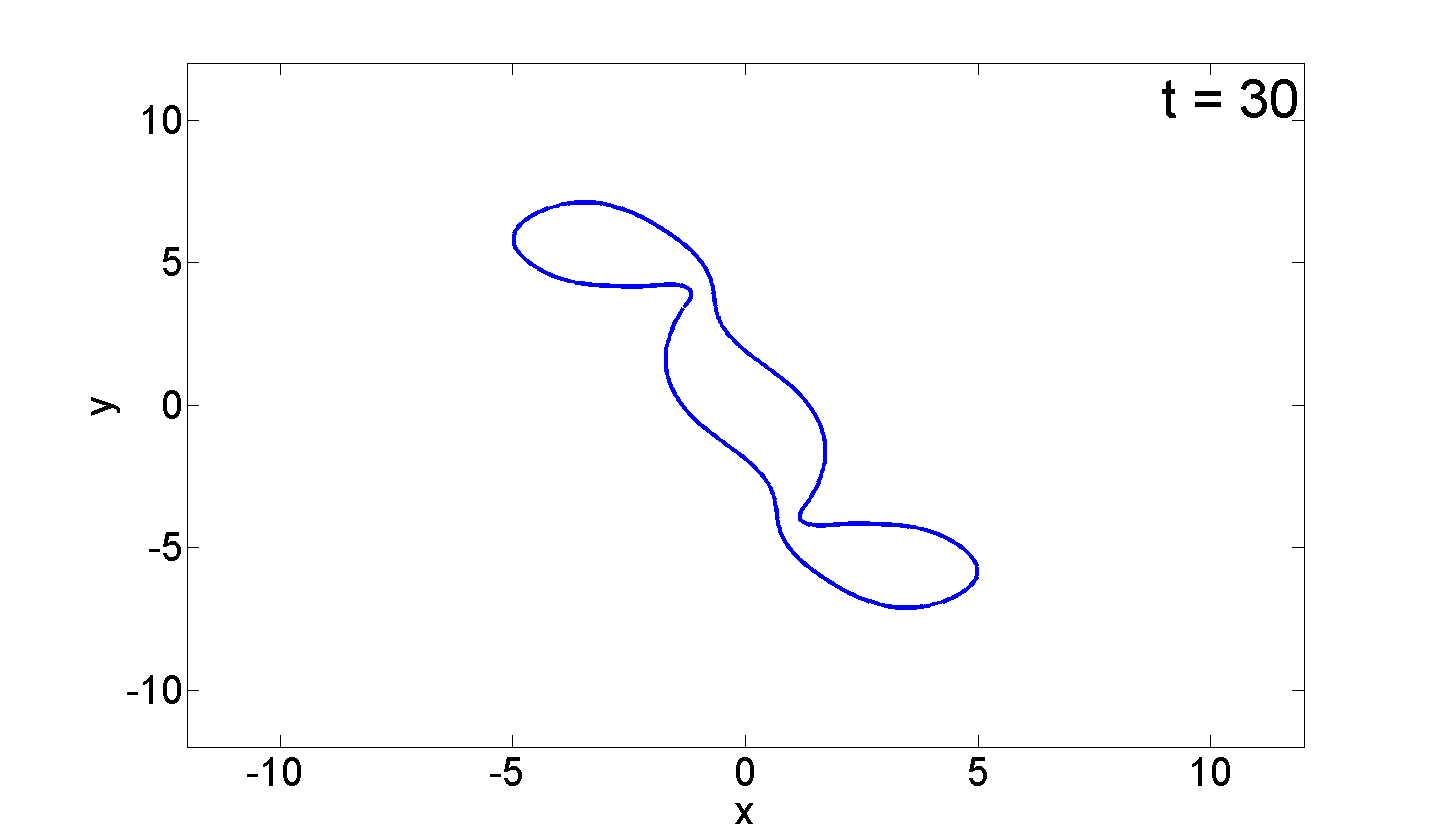}
\includegraphics[width=0.45\textwidth]{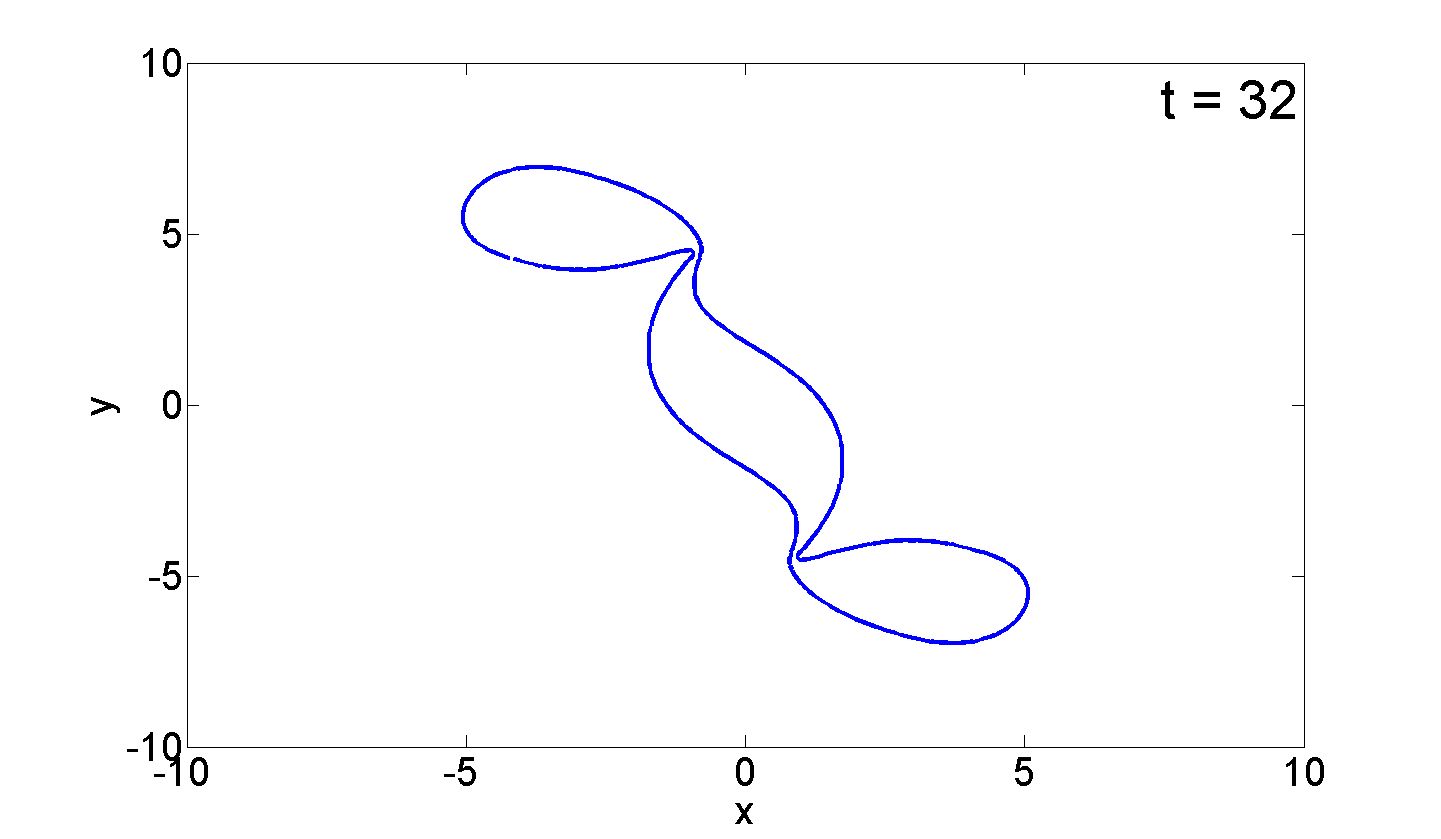}\\
\label{Elipse110}
\caption{Semiaxes 1-10,\ $ t=10,20,30,32$}
\end{figure}



The numerical simulations were performed using a contour dynamics algorithm. The interface is discretized by a 512 point grid and the interface evolves in time with a tangential component such that it maintains  $\vert \partial_x z \vert$ constant in space. In addition, derivatives are computed by a Fourier method with an exponential filter. Evolution in time is carried over by an adaptive step Runge-Kutta 45 pair.












\section{Rotating solutions}


	The purpose of this section is to show some differences between the rotating solutions of the vortex patch and the $\alpha$-patch problem.
	
	It is well known that the circular patch is a steady solution of all those problems, as can be easily seen by noticing that the operator $ (- \Delta)^{1 - \frac{\alpha}{2}}$ is a radial multiplier, and hence, it maps radial functions into radial functions. Since the velocity field is the skew gradient of $(- \Delta)^{1 - \frac{\alpha}{2}} \theta$, where $\theta $ denotes the step function of the patch, the velocity field is tangent to the level sets of $\theta$, leaving them invariant.

	
	In addition to the circular patch, for the vortex patch problem, there is a family of rotating solutions, the so called V-states. These are patches which only rotate around their center of mass with constant angular velocity and which have $n$-fold symmetry. The other remarkable property of V-states is that they have $n$-axes of symmetry. Some numerical results about the shape of V-states can be found in  \cite{Deem-Zabusky:vortex-waves-stationary} and \cite{Wu-Overman-Zabusky:steady-state-Euler-2d} and references therein. The existence of V-states can be proven by a bifurcation analysis from the circle. However, there is no explicit formula for the V-states with more than two axes of symmetry, which correspond to ellipses. The fact that the ellipses are rotating solutions for the vortex patch problem is a classical result \cite{Lamb:hydrodynamics}. However, we will show that ellipses are not rotating $\alpha$-patches, for $\alpha  > 0$.

\begin{thm}
The ellipses are not rotating patches for the equation \eqref{Ecuacion-alpha-patch} with $ 0<\alpha<2$.
\end{thm}

\begin{proof}
Assuming that the center of mass of the patch is the origin, the condition that the patch rotates with constant angular velocity is equivalent to the existence of an $\Omega$ such that:
	
\begin{align*}
\partial_t z(x,t) = i \Omega z(x,t),
\end{align*}
	
for all $x \in [0, 2 \pi]$,	where $z(x,t)$ is a complex parametrization of the boundary.
	
Since the evolution of the patch is completely determined by the normal component of the velocity at the boundary, it is enough to check that the previous identity is not satisfied in the normal direction.


Plugging $z(x,t) = (R_1 \cos(x), R_2 \sin(x))$ into the integral formula for the velocity and taking the inner product of both sides of the previous equation with $\partial_x z(x)^\perp$, we are led, after some lengthy manipulations, to check that
\begin{align*}
F(x) = C_{R_1,R_2,\alpha} \int_0 ^{2 \pi} \frac{\sin(\frac{x-y}{2}) \cos(\frac{x-y}{2})}{(\sin^2(\frac{x - y}{2}))^\frac{\alpha}{2}(1 + R \cos(x + y))^\frac{\alpha}{2}} dy
\end{align*}

where $R = \frac{R_1 ^2 - R_2 ^2 }{R_1 ^2 + R_2 ^2 },0<R<1$, and

\begin{align*} G(x) =  \sin(2x)\end{align*}

are not proportional.\\


We will do so by showing that $H(x)= \frac{F(x)}{G(x)}$ has different limits at $x = 0$ and $x = \frac{\pi}{2}$. $F$ vanishes both at $x = 0$ and $x = \frac{\pi}{2}$ due to the fact that the integrand is odd with respect to those points, and hence, $H(x)$ is continuous at $x = 0$ and $x = \frac{\pi}{2}$ with limits

\begin{align*}
H(0 ) = \frac{F'(0)}{2}, \quad H\left(\frac{\pi}{2}\right) =  -\frac {F'(\frac{\pi}{2})}{2}.
\end{align*}
We will check that $H(0) - H(\frac{\pi}{2}) = \frac{1}{2} (F'(0) + F'( \frac{\pi} {2})) > 0$.\\


Taking derivatives with respect to $x$ and evaluating at $x = 0$ and $ x  = \frac{\pi}{2}$, we obtain


\footnotesize
\begin{align*}
& F'(0) + F' \left( \frac{\pi}{2}\right)   \\
& = C_{\alpha,R} \int_0 ^{2 \pi} \sin\left(\frac{y}{2}\right)^{1 - \alpha} \cos \left(\frac{y}{2}\right) \left ( \frac{R \sin( - y) }{ (1 - R \cos(-y))^{\frac{\alpha}{2} + 1}} - \frac{R \sin(  y) }{ (1 - R \cos(y))^{\frac{\alpha}{2} + 1}} \right)    \\
& + C_{\alpha,R} \int_0 ^{2 \pi} \sin\left(\frac{y}{2}\right)^{1 - \alpha} \cos \left(\frac{y}{2}\right) \left( \frac{R \sin( \pi - y) }{ (1 - R \cos( \pi -y))^{\frac{\alpha}{2} + 1}} - \frac{R \sin(\pi +  y) }{ (1 - R \cos( \pi + y))^{\frac{\alpha}{2} + 1}} \right) dy \\
= & C_{\alpha,R} \int_0  ^\pi  \sin\left(\frac{y}{2}\right)^{2 - \alpha} \cos \left(\frac{y}{2}\right)^2
 \left(\frac{1}{ (1 - R \cos(y))^{\frac{\alpha}{2} + 1}} - \frac{1 }{ (1 + R \cos(y))^{\frac{\alpha}{2} + 1}} \right)  dy\\
 = & C_{\alpha,R} \int_0 ^\frac{\pi}{2} \sin\left(\frac{y}{2}\right)^{2 - \alpha} \cos \left(\frac{y}{2}\right)^2\left(\frac{1}{ (1 - R \cos(y))^{\frac{\alpha}{2} + 1}} - \frac{1 }{ (1 + R \cos(y))^{\frac{\alpha}{2} + 1}}\right)dy    \\
 & - C_{\alpha,R}\int_0 ^\frac{\pi}{2} \cos\left(\frac{y}{2}\right)^{2 - \alpha} \sin \left(\frac{y}{2}\right)^2 \left(\frac{1}{ (1 - R \cos(y))^{\frac{\alpha}{2} + 1}} - \frac{1 }{ (1 + R \cos(y))^{\frac{\alpha}{2} + 1}} \right)dy   \\
= & C_{\alpha,R}\int_0 ^\frac{\pi}{2} \left(  \sin\left(\frac{y}{2}\right)^{2 - \alpha} \cos \left(\frac{y}{2}\right)^2 -  \cos\left(\frac{y}{2}\right)^{2 - \alpha} \sin \left(\frac{y}{2}\right)^2 \right) \left( \frac{1}{ (1 - R \cos(y))^{\frac{\alpha}{2} + 1}} - \frac{1 }{ (1 + R \cos(y))^{\frac{\alpha}{2} + 1}} \right)dy \\
= &  C_{\alpha,R}\int_0 ^\frac{\pi}{2} \sin\left(\frac{y}{2}\right)^{2 - \alpha} \cos\left(\frac{y}{2}\right)^{2 - \alpha} \left( \cos\left(\frac{y}{2}\right)^\alpha - \sin\left(\frac{y}{2}\right)^\alpha \right) \left( \frac{1}{ (1 - R \cos(y))^{\frac{\alpha}{2} + 1}} - \frac{1 }{ (1 + R \cos(y))^{\frac{\alpha}{2} + 1}} \right) dy,
\end{align*}
\normalsize
where the constant $C_{\alpha,R}$ might change from line to line. Finally, the last integral is strictly positive since $ \cos(x) > \sin(x) > 0 $ for $ 0 \leq x \leq \frac{\pi}{4}$.\\

\end{proof}

\section{Loss of convexity}
This section is devoted to prove the following theorem:

\begin{thm}
 For every $0 < \alpha < 2$, there exists a solution of the $\alpha$-patch equation \eqref{Ecuacion-alpha-patch} that begins convex, and after a finite time it is no longer convex. Moreover, there exists a solution of the vortex patch problem that has the same property.
\end{thm}

We start with some technical lemmas:
\begin{lemma}
 Let
\begin{align*}
 k_C(x)  = &4\pi^2((\pi^4 - 3x^4)\cos(C-x) - x(\pi^2 - x^2)^2\sin(C-x)), \quad C \in \{0.15, 0.45\} \\
d_1(x)  = &-4\pi^2x \\
d_2(x)  = &-4\pi^2(\pi^4 - 3x^4) \\
d_3(x)  = &-8\pi^2 x (3\pi^6 - 10 \pi^4 x^2 + 3 \pi^2 x^4 + 6 x^6) \\
d_4(x)  = &-8\pi^2 (3 \pi^{10} - 6 \pi^8 x^2 - 58 \pi^6 x^4 + 132 \pi^4 x^6 - 45 \pi^2 x^8 - 30 x^{10}) \\
d_5(x)  = &-16 \pi^2 x(15 \pi^{12} - 170 \pi^{10}x^2 + 264 \pi^8 x ^4 + 300 \pi^6 x^6 - 765 \pi^4 x^8 + 270 \pi^2 x^{10} + 90 x^{12})\\
d_6(x)  = &-16 \pi^2 (15 \pi^{16} - 270 \pi^{14}x^2 - 1005 \pi^{12} x ^4 + 7102 \pi^{10} x^6 - 9645 \pi^8 x^8 + 930 \pi^{6} x^{10} \\
& + 8505 \pi^{4} x^{12} - 3150 \pi^{2} x^{14} - 630 x^{16})
\end{align*}

Then, the following inequalities are true:
\begin{align*}
 k_C(x) > 0 & \text{ for every } x \in [-\pi,\pi] & \\
 d_1(x) > 0 & \text{ for every } x \in [-\pi,-\pi + \epsilon] \text{ and } d_1(x) < 0 & \text{ for every } x \in [\pi-\epsilon,\pi]\\
 d_2(x) > 0 & \text{ for every } x \in [-\pi,-\pi+\epsilon] \cup [\pi-\epsilon,\pi] &\\
 d_3(x) > 0 & \text{ for every } x \in [-\pi,-\pi + \epsilon] \text{ and } d_3(x) < 0 & \text{ for every } x \in [\pi-\epsilon,\pi]\\
 d_4(x) > 0 & \text{ for every } x \in [-\pi,-\pi+\epsilon] \cup [\pi-\epsilon,\pi] &\\
 d_5(x) > 0 & \text{ for every } x \in [-\pi,-\pi + \epsilon] \text{ and } d_5(x) < 0 & \text{ for every } x \in [\pi-\epsilon,\pi]\\
 d_6(x) > 0 & \text{ for every } x \in [-\pi,-\pi+\epsilon] \cup [\pi-\epsilon,\pi] &\\
\end{align*}

where $\epsilon = 0.0078125 = \frac{1}{128}$.

\end{lemma}
\begin{proof}
 The proof is computer-assisted and the codes can be found in the supplementary material. To compute the positiveness or negativeness, the program checks recursively halving the interval that is validating, until it either achieves a sign condition or the interval is less than a minimum size (in our case $2 \cdot 10^{-10}$). The program passed all the checks without problems.
\end{proof}

From the previous lemma follow the next corollaries:
\begin{cor}
\label{coronlyzerok}
 Let
\begin{align*}
 z_1(x) & = 2 e^{1- \frac{1}{1-\left(\frac{x}{\pi}\right)^{2}}} - 1\\
z_2(x) & = \sin(x-C),
\end{align*}
be defined on $[-\pi,\pi]$ and let $C = 0.45$ or $C = 0.15$. Then the curvature
\begin{align*}
 K(x) = (-\partial_{xx}z_1(x) \partial_{x}z_2(x) + \partial_{xx} z_2(x) \partial_{x}z_1(x))/(\partial_{x}z_1(x)^2 + \partial_{x}z_2(x)^2)^{3/2}
\end{align*}
 is only zero at $x = \pi$.
\end{cor}
\begin{proof}
 The proof follows from the fact that the denominator of $K(x)$ is never zero and the numerator is equal to $k_C(x)e^{1- \frac{1}{1-\left(\frac{x}{\pi}\right)^{2}}}$.
\end{proof}

\begin{cor}
\label{corbddderivatives}
 Let $z_1(x), z_2(x)$ be defined as in the previous corollary and let $\epsilon = \frac{1}{128}$. Then, for every $x \in [-\pi,-\pi + \epsilon]$, $\partial_x^k z_1(x)$ belongs to the convex hull of $\partial_x^{k}z_1(-\pi)$ and $\partial_x^{k}z_1(-\pi+\epsilon)$ and  for every $x \in [\pi-\epsilon,\pi]$, $\partial_x^k z_1(x)$ belongs to the convex hull of $\partial_x^{k} z_1(\pi-\epsilon)$ and $\partial_x^{k} z_1(\pi)$ when $k=0,1,2,3,4,5$.
\end{cor}


\begin{proof}
 The proof is immediate after noticing that
\begin{align*}
\partial_x^{k} z_1(x) = d_k(x) e^{1- \frac{1}{1-\left(\frac{x}{\pi}\right)^{2}}} / (x^2 - \pi^2)^{2k}.
\end{align*}
\end{proof}

\begin{proofthm}
 The proof of the theorem is computer-assisted and can be found in the supplementary material. Let us explain the necessary reductions of the problem to a computationally tractable one.

The general strategy is the following: we will start with an initial condition such that the curvature is zero at exactly one point $x$ and positive otherwise. Then, we will rigorously prove that the derivative of the curvature at that point is different than zero. Since the $\alpha$-patch equations are time reversible, we can go a sufficiently small time forward and backward so that we are able to show that there exists a solution that starts being convex and loses the convexity after a finite time by choosing the initial jump of temperatures across the boundary to be either $+2 \pi$ or $- 2\pi$.

\begin{rem}
By the time reversibility of the equations, this also shows that there are initial conditions that are non convex and in finite time they become convex.
\end{rem}

More precisely, we will show the following: if $z_1$ and $z_2$ are defined as in Corollary \ref{coronlyzerok}, then:

\begin{itemize}
 \item If $C = 0.15$, then $K_t(x) < 0$ if $\alpha \in [0,0.08]$ and $K_t(x) > 0$ if $\alpha \in [0.09,2.00)$.
 \item If $C = 0.45$, then $K_t(x) < 0$ if $\alpha \in [0,0.09]$ and $K_t(x) > 0$ if $\alpha \in [0.10,2.00)$.
\end{itemize}

We now detail the formulas needed for the computation of the time derivative of the curvature.

\begin{align*}
 \partial_t \left(\frac{\langle z_{xx}(x), z_{x}^{\perp}(x)\rangle}{|z_{x}|^3}\right)
& = - \left(\frac{\langle \partial_t z_{x}(x), z_{xx}^{\perp}(x)\rangle}{|z_{x}|^3}\right)
 + \left(\frac{\langle \partial_t z_{xx}(x), z_{x}^{\perp}(x)\rangle}{|z_{x}|^3}\right) \\
& + \underbrace{\langle z_{xx}(x), z_{x}^{\perp}(x)\rangle \partial_t \left(\frac{1}{|z_{x}|^3}\right)}_{0},
\end{align*}

where the last term is zero since the curvature is zero at the point at which we are evaluating the expression. For the vortex patch equation, taking derivatives from the contour equation \eqref{Ecuation-Vortex-Patch} and integrating by parts yield

\begin{align*}
z_t(x) & = \int \log(|z(x)-z(x-y)|)z_{x}(x-y)dy \\
z_{xt}(x) & = \int \log(|z(x)-z(x-y)|)z_{xx}(x-y)dy\\
& + \int \frac{\langle z(x)-z(x-y),z_{x}(x)-z_{x}(x-y) \rangle}{|z(x)-z(x-y)|^{2}} z_x(x-y)dy \\
& = -\int \frac{\langle z(x)-z(x-y),z_{x}(x-y) \rangle}{|z(x)-z(x-y)|^{2}} (z_x(x) - z_x(x-y))dy \\
& + \int \frac{\langle z(x)-z(x-y),z_{x}(x)-z_{x}(x-y) \rangle}{|z(x)-z(x-y)|^{2}} z_x(x-y)dy \\
z_{xxt}(x) & = -\int \frac{\langle z(x)-z(x-y),z_{x}(x-y) \rangle}{|z(x)-z(x-y)|^{2}} (z_{xx}(x) - z_{xx}(x-y))dy \\
& + 2\int \frac{\langle z(x)-z(x-y),z_{x}(x)-z_{x}(x-y) \rangle}{|z(x)-z(x-y)|^{2}} z_{xx}(x-y)dy \\
& + \int \frac{|z_x(x)-z_{x}(x-y)|^{2}}{|z(x)-z(x-y)|^{2}} z_{x}(x-y)dy \\
& + \int \frac{\langle z(x)-z(x-y),z_{xx}(x)-z_{xx}(x-y) \rangle}{|z(x)-z(x-y)|^{2}} z_{x}(x-y)dy \\
& - 2\int \frac{(\langle z(x)-z(x-y),z_{x}(x)-z_{x}(x-y) \rangle)^{2}}{|z(x)-z(x-y)|^{4}} z_{x}(x-y)dy, \\
\end{align*}

whereas for the $\alpha$-patch the derivatives are given by

\begin{align*}
 z_t(x) & = c_{\al} \int \frac{z_{x}(x) - z_{x}(x-y)}{|z(x)-z(x-y)|^{\al}}dy \\
z_{xt}(x) & = c_{\al} \int \frac{z_{xx}(x) - z_{xx}(x-y)}{|z(x)-z(x-y)|^{\al}}dy \\
& - c_{\al} \al \int \frac{z_{x}(x) - z_{x}(x-y)}{|z(x)-z(x-y)|^{2+\al}}\langle z_{x}(x) - z_{x}(x-y), z(x) - z(x-y)\rangle dy \\
z_{xxt}(x) & = c_{\al} \int \frac{z_{xxx}(x) - z_{xxx}(x-y)}{|z(x)-z(x-y)|^{\al}}dy \\
& - 2 c_{\al} \al \int \frac{z_{xx}(x) - z_{xx}(x-y)}{|z(x)-z(x-y)|^{2+\al}}\langle z_{x}(x) - z_{x}(x-y), z(x) - z(x-y)\rangle dy \\
& - c_{\al} \al \int \frac{z_{x}(x) - z_{x}(x-y)}{|z(x)-z(x-y)|^{2+\al}}|z_{x}(x) - z_{x}(x-y)|^{2} dy \\
& + c_{\al} \al (\al + 2) \int \frac{z_{x}(x) - z_{x}(x-y)}{|z(x)-z(x-y)|^{4+\al}}(\langle z_{x}(x) - z_{x}(x-y), z(x) - z(x-y)\rangle)^{2} dy \\
& - c_{\al} \al \int \frac{z_{x}(x) - z_{x}(x-y)}{|z(x)-z(x-y)|^{2+\al}}\langle z_{xx}(x) - z_{xx}(x-y), z(x) - z(x-y)\rangle dy. \\
\end{align*}


 Thus, in this case we need to validate the sign of a quantity which can be represented as
\begin{align*}
 c_{\al} \int F(\al,x) dx = c_{\al} I(\al).
\end{align*}

The problem that we encounter is that for small values of $\alpha$ tending to zero, $c_{\al}$ tends to $\infty$ and $I(\al)$ tends to zero. To overcome this difficulty, we will distinguish four regimes of $\alpha$:
\begin{enumerate}
 \item $\alpha = 0$: the vortex patch equation.
\item $\alpha \leq \alpha_{cr}$: the small $\alpha$ region.
\item $\alpha_{cr} \leq \alpha \leq \alpha_{br}$: the big $\alpha$ region.
\item $\alpha_{br} \leq \alpha  < 2$: the very big $\alpha$ region.
\end{enumerate}

In the first case, we are left to validate the sign of an integral which is independent of $\alpha$. We will show that this sign is negative for the functions considered. In the second case, using the fact that for $\alpha$ strictly positive, $0 < c_{\al} < \infty$ and that $I(0) = 0$, it is enough to show that $DI(\alpha) = \partial_{\alpha}I(\alpha) < 0$. This is the sign condition that we will validate in the second region. In the third and fourth regions, we will validate $I(\al) < 0$ or $I ( \al) > 0$ depending on the range of $\alpha$. We chose $\alpha_{cr} = 0.04$ and $\alpha_{br} = 1.95$. We remark that both $I(\alpha)$ and $DI(\alpha)$ have the same sign if multiplied by $|z_x|^3$. Therefore, in order to avoid the division, we will compute the values of $I(\alpha)|z_x|^3$ and $DI(\alpha)|z_x|^3$.

The integral that we need to calculate in the small $\alpha$ region is

\begin{align*}
- \left(\frac{\langle \partial_{\alpha} \partial_t z_{x}(x), z_{xx}^{\perp}(x)\rangle}{c_{\alpha}|z_{x}|^3}\right)
 + \left(\frac{\langle \partial_{\alpha} \partial_t z_{xx}(x), z_{x}^{\perp}(x)\rangle}{c_{\alpha}|z_{x}|^3}\right) = DI(\alpha)
\end{align*}
where the derivatives with respect to $\alpha$ are given by

\begin{align*}
&\partial_{\alpha} \left(\frac{z_{xt}(x)}{c_{\alpha}}\right)  = - \int \frac{z_{x}(x) - z_{x}(x-y)}{|z(x)-z(x-y)|^{2+\al}}\langle z_{x}(x) - z_{x}(x-y), z(x) - z(x-y)\rangle dy \\
& - \int \frac{z_{xx}(x) - z_{xx}(x-y)}{|z(x)-z(x-y)|^{\al}}\log(|z(x)-z(x-y)|) dy \\
& + \al \int \frac{z_{x}(x) - z_{x}(x-y)}{|z(x)-z(x-y)|^{2+\al}}\langle z_{x}(x) - z_{x}(x-y), z(x) - z(x-y)\rangle \log(|z(x)-z(x-y)|) dy \\
& \partial_{\alpha} \left(\frac{z_{xxt}(x)}{c_{\alpha}}\right)   =  - 2  \int \frac{z_{xx}(x) - z_{xx}(x-y)}{|z(x)-z(x-y)|^{2+\al}}\langle z_{x}(x) - z_{x}(x-y), z(x) - z(x-y)\rangle dy \\
& -   \int \frac{z_{x}(x) - z_{x}(x-y)}{|z(x)-z(x-y)|^{2+\al}}|z_{x}(x) - z_{x}(x-y)|^{2} dy \\
& +  (2 \al + 2) \int \frac{z_{x}(x) - z_{x}(x-y)}{|z(x)-z(x-y)|^{4+\al}}(\langle z_{x}(x) - z_{x}(x-y), z(x) - z(x-y)\rangle)^{2} dy \\
& -   \int \frac{z_{x}(x) - z_{x}(x-y)}{|z(x)-z(x-y)|^{2+\al}}\langle z_{xx}(x) - z_{xx}(x-y), z(x) - z(x-y)\rangle dy. \\
& - \int \frac{z_{xxx}(x) - z_{xxx}(x-y)}{|z(x)-z(x-y)|^{\al}}\log(|z(x)-z(x-y)|)dy \\
& + 2  \al \int \frac{z_{xx}(x) - z_{xx}(x-y)}{|z(x)-z(x-y)|^{2+\al}}\langle z_{x}(x) - z_{x}(x-y), z(x) - z(x-y)\rangle \log(|z(x)-z(x-y)|) dy \\
& +  \al \int \frac{z_{x}(x) - z_{x}(x-y)}{|z(x)-z(x-y)|^{2+\al}}|z_{x}(x) - z_{x}(x-y)|^{2}  \log(|z(x)-z(x-y)|) dy \\
& - \al (\al + 2) \int \frac{z_{x}(x) - z_{x}(x-y)}{|z(x)-z(x-y)|^{4+\al}}(\langle z_{x}(x) - z_{x}(x-y), z(x) - z(x-y)\rangle)^{2} \log(|z(x)-z(x-y)|) dy \\
& +  \al \int \frac{z_{x}(x) - z_{x}(x-y)}{|z(x)-z(x-y)|^{2+\al}}\langle z_{xx}(x) - z_{xx}(x-y), z(x) - z(x-y)\rangle \log(|z(x)-z(x-y)|) dy. \\
\end{align*}


There are some technical issues while calculating $I(\al)$ when $\al$ goes to 2 in the very big region, namely that $I(\al)$ is not well defined for $\al = 2$. To overcome this difficulty, we will extend $I(\al)$ to a function $\tilde{I}(\al)$ such that $I(\al) = \tilde{I}(\al)$ for $\al < 2$ and $\tilde{I}(2)$ is well defined.

In order to carry out this extension, we just notice that

\begin{align*}
\int \frac{z_{x}(x)-z_{x}(x-y)}{|z(x)-z(x-y)|^{\al}}dy =
\int \left(\frac{z_{x}(x)-z_{x}(x-y)}{|z(x)-z(x-y)|^{\al}} - \frac{z_{xx}(x) \sgn(y)}{\left|2\tan\left(\frac{y}{2}\right)\right|^{\al-1}}\right)dy.
\end{align*}

Then,

\begin{align*}
 \tilde{I}(\al) =  - \left(\frac{\langle \partial_t \tilde{z}_{x}(x), z_{xx}^{\perp}(x)\rangle}{c_{\al}|z_{x}|^3}\right)
 + \left(\frac{\langle \partial_t \tilde{z}_{xx}(x), z_{x}^{\perp}(x)\rangle}{c_{\al}|z_{x}|^3}\right),
\end{align*}

where

\begin{align*}
\tilde{z}_{xt} & = c_{\al} \int \frac{z_{xx}(x) - z_{xx}(x-y)}{|z(x)-z(x-y)|^{\al}}dy - c_{\al} \frac{z_{xxx}(x)}{|z_{x}(x)|^{\al}}\int \frac{\sgn(y)}{\left|2\tan\left(\frac{y}{2}\right)\right|^{\al-1}}dy \\
& - c_{\al} \al \int \frac{z_{x}(x) - z_{x}(x-y)}{|z(x)-z(x-y)|^{2+\al}}\langle z_{x}(x) - z_{x}(x-y), z(x) - z(x-y)\rangle dy \\
& + c_{\al} \al \frac{z_{xx}(x)}{|z_x(x)|^{2+\al}}\langle z_{xx}(x), z_x(x)\rangle \int \frac{\sgn(y)}{\left|2\tan\left(\frac{y}{2}\right)\right|^{\al-1}}dy \\
\tilde{z}_{xxt} & =  c_{\al} \int \frac{z_{xxx}(x) - z_{xxx}(x-y)}{|z(x)-z(x-y)|^{\al}}dy -c_{\al}\frac{z_{xxxx}(x)}{|z_x(x)|^{\al}} \int \frac{\sgn(y)}{\left|2\tan\left(\frac{y}{2}\right)\right|^{\al-1}}dy\\
& - 2 c_{\al} \al \int \frac{z_{xx}(x) - z_{xx}(x-y)}{|z(x)-z(x-y)|^{2+\al}}\langle z_{x}(x) - z_{x}(x-y), z(x) - z(x-y)\rangle dy \\
& + 2 c_{\al} \al \frac{z_{xxx}(x)}{|z_x(x)|^{2+\al}}\langle z_{xx}(x), z_x(x)\rangle \int \frac{\sgn(y)}{\left|2\tan\left(\frac{y}{2}\right)\right|^{\al-1}}dy\\
& - c_{\al} \al \int \frac{z_{x}(x) - z_{x}(x-y)}{|z(x)-z(x-y)|^{2+\al}}|z_{x}(x) - z_{x}(x-y)|^{2} dy \\
& + c_{\al} \al \frac{z_{xx}(x)}{|z_x(x)|^{2+\al}}|z_{xx}(x)|^{2} \int \frac{\sgn(y)}{\left|2\tan\left(\frac{y}{2}\right)\right|^{\al-1}}dy\\
& + c_{\al} \al (\al + 2) \int \frac{z_{x}(x) - z_{x}(x-y)}{|z(x)-z(x-y)|^{4+\al}}(\langle z_{x}(x) - z_{x}(x-y), z(x) - z(x-y)\rangle)^{2} dy \\
& - c_{\al} \al (\al + 2) \frac{z_{xx}(x)}{|z_x(x)|^{4+\al}}(\langle z_{xx}(x), z_x(x)\rangle)^{2} \int \frac{\sgn(y)}{\left|2\tan\left(\frac{y}{2}\right)\right|^{\al-1}}dy\\
& - c_{\al} \al \int \frac{z_{x}(x) - z_{x}(x-y)}{|z(x)-z(x-y)|^{2+\al}}\langle z_{xx}(x) - z_{xx}(x-y), z(x) - z(x-y)\rangle dy \\
& + c_{\al} \al \frac{z_{xx}(x)}{|z_x(x)|^{2+\al}}\langle z_{xxx}(x), z_x(x)\rangle \int \frac{\sgn(y)}{\left|2\tan\left(\frac{y}{2}\right)\right|^{\al-1}}dy.\\
\end{align*}

We remark that $\tilde{I}(2)$ is well defined. 
Again, since $c_{\al} > 0$ for $\al < 2$, we will only care about the sign of $I(\al)|z_{x}|^{3}$.


The algorithm for the computation of the integral is as follows: we define a structure called \texttt{ParameterSet}, which encapsulates all the necessary information about the parameters and the information needed by the integration procedures in order to perform the validation for those parameters. More precisely, a ParameterSet contains:
\begin{itemize}
\item Two intervals, \texttt{Left} and \texttt{Right}, which set the limits for the bounded and singularity regions (i.e. singularity $ = [\text{Left},\text{Right}]$, bounded $= [-\pi,\pi] \setminus$ singularity). In our proof, Left $=-\frac{1}{128}$, Right = $\frac{1}{128}$. See below for the precise definition and the integration procedures in these regions.
\item Two doubles, \texttt{AbsTol} and \texttt{RelTol}, which limit the precision up to which the integrals are computed. In our proof, AbsTol $=$ RelTol $= 10^{-6}$.
\item One interval, \texttt{$\alpha$}, which is the interval in the parameter space we are calculating.
\item One interval, \texttt{$C$}, which corresponds to the different initial conditions.	
\end{itemize}

We will use a queue (implemented using the Standard Template Library (STL) \texttt{Queue}), in which we store all the ParameterSets to be computed. While the queue is not empty, we take the top element, pop it and give an enclosure of integral we are validating for this region. Three different things can happen:
\begin{itemize}
\item The enclosure is positive.
\item The enclosure is negative.
\item We can not say anything about its positivity.
\end{itemize}

In the first two cases, the result is output to its corresponding file (one for the regions for which the aforementioned integral is positive, another for the ones for which it is negative). In the third case, the ParameterSet is split into other narrower ParameterSets which are pushed in the queue. This splitting is only done if the diameter of $\alpha$ is bigger than a given threshold, which in our case was set to $5 \cdot 10^{-6}$.


We now describe how to perform the integral over the bounded region.  The bounded part is calculated using a Gauss-Legendre quadrature of order 2, given by
 \begin{align*}
 \int_{a}^{b} f(\eta) d\eta &  \in  \frac{b-a}{2}\left(f\left(\frac{b-a}{2}\frac{\sqrt{3}}{3} + \frac{b+a}{2}\right)+f\left(-\frac{b-a}{2}\frac{\sqrt{3}}{3} + \frac{b+a}{2}\right)\right) \\
 & +\frac{1}{4320}(b-a)^{5}f^{4}([a,b]).
 \end{align*}

 Whenever the result does not satisfy some tolerance requirements in the form of having absolute or relative (with respect to the volume of the integration region) width smaller than the two constants AbsTol and RelTol the integration domain is split by the midpoint and we call the integrator again with the new two subdomains recursively. Otherwise we will save it and add it to the total. We also limit the depth of the levels of splitting in order to prevent infinite loops or stack overflows because of too stringent tolerances since the uncertainty of the parameters might yield wide enclosures of the integral even with infinite precision. In our case, the maximum number of subdivisions was 13, totalling a maximum number of subintervals equal to $2^{13}$.


Regarding the singularity region, we remark that if we try to evaluate the (singular) integrand directly, we get expressions of the type $\frac{0}{0}$ which can only be enclosed in unbounded intervals. In order to avoid this phenomenon we will bound the absolute value of the integrand and consider the contribution of the singularity region as a residual error. By virtue of the mean value theorem, we will substitute any expressions of the form $|(\partial_{x}^{k}z(a) - \partial_{x}^{k} z(b))|$ by the interval $|(a-b)||\partial_{x}^{k+1} z([a,b])|$, and these derivatives can be easily enclosed only by having to evaluate at the endpoints because of the monotonicity proved in Corollary \ref{corbddderivatives}. After estimating the integrals in this way, we are left with the task of having to bound quantities of the form 

\begin{align*}
\int C|x|^{0}dx, \quad  \int C|x|^{1-\alpha}dx,  \int C_1|x|^{1-\alpha}dx + \int C_2 |\log(C_3 x)| |x|^{1-\alpha} dx, \text{ or } \int C|x|^{2-\alpha}dx
\end{align*}
depending if we are in the first, second, third or fourth regime of $\alpha$ respectively, for some constants (intervals) $C,C_1,C_2,C_3$. These integrals can be computed explicitly in closed form. In practice, due to the exponential nature of $z_1$, the contribution of these integrals is extremely small compared to the true result.

We ran the computation in parallel (every core was allocated a different initial region) over 16 cores by splitting the regions of $\alpha$ on an Intel i5 processor with 4 GB of RAM. The total combined computation time was roughly   120 hours.

\end{proofthm}

\section*{Acknowledgements}

AC, DC, JGS and AMZ were partially supported by the grant MTM2011-26696 (Spain), grant StG-203138CDSIF of
the ERC and ICMAT Severo Ochoa project SEV-2011-0087. AC was partially supported by the ERC grant 307179-GFTIPFD. We are grateful to the Instituto de Ciencias Matem\'aticas for computing facilities (Cluster Odisea).

\bibliographystyle{abbrv}
\bibliography{references}

\begin{tabular}{ll}
\textbf{Angel Castro} & \textbf{Diego C\'ordoba} \\
{\small Instituto de Ciencias Matem\'aticas} & {\small Instituto de Ciencias Matem\'aticas}\\
{\small Universidad Aut\'onoma de Madrid} & {\small Consejo Superior de Investigaciones Cient\'ificas}\\
{\small C/ Nicolas Cabrera, 13-15, 28049 Madrid, Spain} & {\small C/ Nicolas Cabrera, 13-15, 28049 Madrid, Spain}\\
{\small Email: angel\underline{  }castro@icmat.es} & {\small Email: dcg@icmat.es}\\
   & \\
\textbf{Javier G\'omez-Serrano} & \textbf{Alberto Mart\'in Zamora}\\
{\small Department of Mathematics} & {\small Instituto de Ciencias Matem\'aticas}\\
{\small Princeton University} & {\small Consejo Superior de Investigaciones Cient\'ificas}\\
{\small 1102 Fine Hall, Washington Rd, } & {\small C/ Nicolas Cabrera, 13-15, 28049 Madrid, Spain}\\
{\small Princeton, NJ 08544, USA} & {}\\
 {\small e-mail: jg27@math.princeton.edu} & {\small Email: alberto.martin@icmat.es}
  \\

\end{tabular}


\end{document}